\documentclass[12pt]{amsart}

\usepackage{amsmath,amssymb}
\usepackage{amsthm}
\usepackage{hyperref}
\usepackage[margin=1.0in]{geometry}

\numberwithin{equation}{section}

\newtheorem{prop}{Proposition}
\newtheorem{lemma}[prop]{Lemma}

\newtheorem{thm}[prop]{Theorem}
\newtheorem{cor}[prop]{Corollary}

\numberwithin{prop}{section}

\theoremstyle{definition}
\newtheorem{defn}[prop]{Definition}

\newtheorem{rmk}[prop]{Remark}

\newcommand{\del}{\partial}
\newcommand{\dt}{\frac{\partial}{\partial t}}
\newcommand{\brs}[1]{\left| #1 \right|}

\newcommand{\gG}{\Gamma}

\newcommand{\gD}{\Delta}
\newcommand{\gd}{\delta}

\newcommand{\gw}{\omega}
\newcommand{\ga}{\alpha}
\newcommand{\gb}{\beta}

\newcommand{\N}{\nabla}
\newcommand{\FF}{\mathcal F}
\newcommand{\WW}{\mathcal W}

\newcommand{\GG}{\mathcal G}

\renewcommand{\bar}[1]{\overline{#1}}

\newcommand{\IP}[1]{\left<#1\right>}

\newcommand{\HH}{\mathcal{H}}

\DeclareMathOperator{\Rc}{Rc}

\DeclareMathOperator{\tr}{tr}

\DeclareMathOperator{\Vol}{Vol}
\DeclareMathOperator{\diam}{diam}

\DeclareMathOperator{\End}{End}

\begin{document}

\title[Structure of collapsing solutions of generalized Ricci flow]{Structure of collapsing solutions of generalized Ricci flow}

\begin{abstract} We derive modified Perelman-type monotonicity formulas for solutions to the generalized Ricci flow equation with symmetry on principal bundles, which lead to rigidity and classification results for nonsingular solutions. 
\end{abstract}

\date{\today}

\author{Steven Gindi}
\address{Surge Hall\\
         University of California\\
         Riverside, CA 92521}
\email{\href{mailto:gindis@ucr.edu}{gindis@ucr.edu}}

\author{Jeffrey Streets}
\address{Rowland Hall\\
         University of California\\
         Irvine, CA 92617}
\email{\href{mailto:jstreets@uci.edu}{jstreets@uci.edu}}

\maketitle

\section{Introduction}

The generalized Ricci flow is a natural coupling of the Ricci flow and the heat equation for a closed three-form, namely
\begin{gather} \label{f:GRF}
\begin{split}
\dt g =&\ -2 \Rc + \tfrac{1}{2} H^2,\\
\dt H =&\ \gD_d H,
\end{split}
\end{gather}
where $H^2_{ij} = H_{ipq} H_j^{pq}$, and $\gD_d$ denotes the Hodge Laplacian with respect to the time varying metric.  This system first arose in the physical theory of renormalization group flows of sigma models (cf. \cite{Polchinski}).  Many interesting properties have been established in recent years, including a gradient formulation \cite{OSW}, basic existence and regularity properties \cite{Streetsexpent}, and a relationship to generalized geometry and $T$-duality \cite{StreetsTdual}.  In \cite{PCFReg, GKRF} it was shown that this flow preserves natural integrability conditions in complex and generalized K\"ahler geometry, coming from the relationship to the pluriclosed flow, described in \S \ref{s:dimredGRF} below.

Our main interest here is to study certain kinds of nonsingular solutions to the generalized Ricci flow and to analyze them by means of energy and entropy functionals.  Dimensional collapse is a common feature in analyzing the long time limits of Ricci flow.  By the compactness theory of Lott \cite{Lott1}, from such a collapsing solution one can extract a limiting solution which exists on a groupoid.  Lott furthermore showed that in dimension $3$, these groupoid limits are essentially solutions of Ricci flow on twisted principal bundles.  By classifying these solutions by means of a modified entropy functional we describe below, Lott in turn classified all type III solutions to Ricci flow on three-manifolds with $\diam (g_t) = O(t^{1/2})$.  Dimensional collapse of course also happens in the generalized Ricci flow, and moreover examples of this behavior unique to the setting of generalized Ricci flow were exhibited by Boling \cite{Boling}.  Following the compactness theory of Lott, one expects the simplest kinds of collapsing limits to be modeled by solutions on twisted principal bundles.  Our main theorem classifies these solutions under certain natural hypotheses.  One particular application of interest is the classification of type III solutions to pluriclosed flow on complex surfaces as a partial test of the second author's geometrization conjecture for complex surfaces \cite{Streetsgeom}, and this will be the subject of a future work.  As in the work of Lott \cite{Lott2}, the classification of these solutions will eventually reduce to this classification result for solutions on twisted principal bundles.  We give an informal statement here, referring to \S \ref{SECTDC} for the notation and \S \ref{s:rigidity} for the precise claims.

\begin{thm} \label{t:mainthm} 
\begin{enumerate}
\item[1)] Let $(P, \bar{g}(\cdot), \bar{H}(\cdot))$ denote a global solution to invariant generalized Ricci flow, and suppose $(P_{\infty}, \bar{g}_{\infty}(\cdot), \bar{H}_{\infty}(\cdot))$ is a subsequential limit flow at infinity.  If $[\bar{H}_{\infty}(\cdot)] = 0$, then $(\bar{g}_{\infty}(\cdot), \bar{H}_{\infty}(\cdot))$ is constant in time, $\bar{H}_{\infty} \equiv 0$ and $\bar{g}_{\infty}$ is a local product metric with $g_{\infty}$ Ricci flat.
\item[2)] Let $(P, \bar{g}(\cdot), \bar{H}(\cdot))$ denote a global solution to invariant generalized Ricci flow with $\dim \GG \leq 2$, and suppose $(P_{\infty}, \bar{g}_{\infty}(\cdot), \bar{H}_{\infty}(\cdot))$ is a subsequential blowdown limit flow at infinity.  Then $H_{\infty} \equiv 0$, $F_{\infty} \equiv 0$, $\det G_{ij}$ is constant, and $\bar{g}$ satisfies the twisted Einstein equations (\ref{f:Ricciexpandersoliton}).
\end{enumerate}
\end{thm}

The recent example of compact nontrivial steady solitons for the pluriclosed flow show that the hypothesis of $[\bar{H}_{\infty}] = 0$ is necessary in the first part of the theorem (cf. Remark \ref{r:steadyrmk}).  Also, the result in the second part of the theorem is slightly more general, applying to cases when $\GG$ is abelian and the torsion vanishes on the fibers, which is implied by $\GG$ being nilpotent of dimension $\leq 2$.  Furthermore, hypotheses of this kind are necessary as illustrated by the case of collapsing type III solutions of pluriclosed flow on Inoue surfaces (cf. Remark \ref{e:Inoue}), noting that the pluriclosed flow in complex geometry is gauge-equivalent to generalized Ricci flow.  In fact, in this setting, the presence of the complex structure gives further rigidity of the limits.  The K\"ahler-Ricci flow occurs as a special case of pluriclosed flow, and the corollary below is related to recent works on collapsing of type III solutions (cf. \cite{TosattiZhang}).

\begin{cor} \label{c:maincor}
\begin{enumerate}
\item[1)] Let $(P, \bar{g}(\cdot), \bar{H}(\cdot), J(\cdot))$ denote a global solution to invariant pluriclosed flow, and suppose $(P_{\infty}, \bar{g}_{\infty}(\cdot), \bar{H}_{\infty}(\cdot), J_{\infty}(\cdot))$ is a subsequential limit flow at infinity.  If $[\bar{H}_{\infty}(\cdot)] = 0$, then $(\bar{g}_{\infty}(\cdot), \bar{H}_{\infty}(\cdot))$ is constant in time, $\bar{H}_{\infty} \equiv 0$ and $\bar{g}_{\infty}$ is a K\"ahler local product metric with $g_{\infty}$ K\"ahler Calabi-Yau.
\item[2)] Let $(P, \bar{g}(\cdot), \bar{H}(\cdot), J(\cdot))$ denote a global solution to invariant pluriclosed flow with $\dim \GG \leq 2$, and suppose $(P_{\infty}, \bar{g}_{\infty}(\cdot), \bar{H}_{\infty}(\cdot), J_{\infty}(\cdot))$ is a subsequential blowdown limit flow at infinity. In addition to the claims of Theorem \ref{t:mainthm} part 2), we have
\begin{itemize}
\item[1)] $J_{\infty}$ is fixed in time and $(\bar{g}_{\infty}(t), {J}_{\infty})$ is Kahler.
\item[2)] $M_{\infty}$ is even dimensional and $J_{\infty}=J_{1} \oplus J_{2}$ on $\mathfrak{G} \oplus TM_{\infty}$.
\end{itemize}
In particular, if $\dim \GG = 1$ then invariant pluriclosed flow has no subsequential limits in the sense of Definition \ref{d:convergence}.
\end{enumerate}
\end{cor}

The key tools for proving Theorem \ref{t:mainthm} and Corollary \ref{c:maincor} are modified energy and entropy functionals for the generalized Ricci flow on twisted principal  bundles.  Among Perelman's key discoveries were the energy and shrinker entropy functionals for Ricci flow \cite{Perelman1}, which are monotonically increasing along Ricci flow and fixed on steady or shrinking solitons, respectively.  Shortly thereafter Feldman-Ilmanen-Ni discovered the expander entropy \cite{FIN}, which is monotonically increasing along Ricci flow and fixed on an expanding soliton.  By applying Hamilton's compactness theory, these functionals all immediately lead to classification results for certain nonsingular solutions which furthermore satisfy a noncollapsing hypothesis.  However, as remarked above, nonsingular solutions of Ricci flow can collapse, and to treat these solutions Lott \cite{Lott2} developed modified energy and entropy functionals for invariant solutions on twisted principal bundles which capture the geometry of the collapsing structure.  As the fibers of these bundles may be noncompact, one cannot simply use the known functionals and incorporate the symmetries.  Rather, one is forced to define integrals over the base space of the bundle, which is assumed compact, and to redefine the conjugate heat equation to preserve a measure on this space.  This is a subtle change which makes their monotonicity a nontrivial question, and indeed there are ``extra'' terms in the evolution of these energies which can have varying signs.

We extend this analysis in several ways.  First, an extension of Perelman's energy functional to generalized Ricci flow was found in \cite{OSW}, while an extension of the expander entropy was discovered by the second author in \cite{Streetsexpent}.  We follow the strategy described above and define dimensionally-reduced versions of these functionals, as well as a shrinker entropy, and derive their evolution equations under generalized Ricci flow.  We note that Lott's work \cite{Lott2} restricted to the case where the structure group of the bundle is abelian.  As remarked above, one of our main intended applications of this work is to the classification of type III solutions of the pluriclosed flow on complex surfaces, where collapse to a one-dimensional space is known to happen \cite{Boling}, with a three-dimensional and hence potentially nonabelian structure group.  For this reason we treat the case of a general nilpotent structure group, and so these evolution equations are extensions of Lott's even in the case of Ricci flow.  The nontrivial Lie bracket introduces terms of varying signs in the different functionals, rendering their application more delicate.

Here is an outline of the rest of this paper.  In \S \ref{SECTDC} we recall fundamental aspects of invariant metrics on principal bundles.  In \S \ref{s:dimredGRF} we derive the dimensionally-reduced evolution equations for generalized Ricci flow.  Next in \S \ref{s:energy} we define the modified energy and entropy functionals and derive their evolution along generalized Ricci flow.  We analyze these evolutions in \S \ref{s:rigidity} to establish the main rigidity results of Theorem \ref{t:mainthm} and Corollary \ref{c:maincor}.

\section{Geometry of invariant metrics on principal bundles} \label{SECTDC}

Our aim is to study blowdown limits of invariant solutions to generalized Ricci flow on principal bundles.  As we are interested in capturing the geometry of solutions along which the fibers collapse, it is useful to recast these flow solutions in terms of a system of PDEs on the base manifold. To do so we will need to describe a correspondence between invariant sections of $TP$ and sections of $\mathcal{E}:=\mathfrak{G}\oplus TM$, where $\mathfrak{G}$ is the adjoint bundle, and build Lie algebroid structures and connections on $\mathcal{E}$.  We discuss convergence of invariant metrics and one-parameter families in preparation for deriving the reduced generalized Ricci flow system in Section \ref{s:dimredGRF}.

\subsection{Setup}

We fix a principal $\mathcal{G}$ bundle $\pi: P\rightarrow M$, and let $\bar{g}$ denote an invariant metric and $\bar{H}$ an invariant three-form on $P$.  Given $\bar{g}$, we let $\bar{A}$ denote the associated connection on $P$ and $\mathcal{E}=\mathfrak{G}\oplus TM$, where $\mathfrak{G}$ is the adjoint bundle. Also let $F\in \wedge^{2}T^{*}M\otimes \mathfrak{G}$ correspond to the equivariant horizontal curvature $\bar{F}:=d\bar{A}+\tfrac{1}{2}[\bar{A},\bar{A}]$.  In some places we will also assume the existence of $\bar{J}$ an invariant complex structure on $P$, and in this case $\bar{g}$ will be compatible with $\bar{J}$, and furthermore 
\begin{align*}
\bar{H} = -d^c \bar{\gw}, 
\end{align*}
where $\gw(X,Y) = g(JX, Y)$.

Given a principal connection $\bar{A}$, we construct an isomorphism $\delta: TP \longrightarrow \pi^{*}\mathcal{E}$ given by $\delta(Z)= \{p,\bar{A}Z \} + \pi_{*}Z,$ where $Z \in T_{p}P$. If $e \in \Gamma(\mathcal{E})$ then $\delta^{-1}\pi^{*}e$ is an invariant section of $TP$, where $\pi^{*}e$ is the associated section of $\pi^{*}\mathcal{E}$. Moreover if $Z$ is an invariant section of $TP$ then it is straightforward to show that $Z= \delta^{-1}\pi^{*}e$ for some unique $e \in \Gamma(\mathcal{E})$. We will typically denote $\delta^{-1}\pi^{*}e$ by $\tilde{e}$.  This correspondence extends naturally to invariant sections of $(\otimes^{k}T^{*}P)\otimes (\otimes^{l}TP) $
and sections of $(\otimes^{k}\mathcal{E^{*}})\otimes (\otimes^{l}\mathcal{E})$.  In particular, by using this identification, for invariant structures $\bar{g}$,  $\bar{H}$ and $\bar{J}$ on $P$, this correspondence respectively yields the fiberwise metric $g_{_{\mathcal{E}}}=G\oplus g$ on $\mathcal{E}$,  $H \in \wedge^{3}\mathcal{E}^{*}$, and $J \in \End \mathcal{E}$ that satisfies $J^{2}=-1$.

\subsection{Lie Algebroid and Differential Structures}

Given the setup of the previous subsection, we now introduce differential and  Lie algebroid structures on $\mathcal{E}$. 

\begin{defn}
For $e_{1}, e_{2} \in \Gamma(\mathcal{E})$. Define $[e_{1},e_{2}]$ to be the unique section of $\mathcal{E}$ that satisfies \[ \widetilde{[e_{1},e_{2}]}= [\tilde{e_{1}},\tilde{e_{2}}]_{Lie},\] where $[,]_{Lie}$ is the Lie bracket on $\Gamma(TP)$. 
\end{defn}

If we define $ \tau: \mathcal{E}\longrightarrow TM$ by $\tau(\eta + v)= v $, where $\eta \in \mathfrak{g}$,  we then have: 

\begin{prop} \label{p:Liealgprop} Let $e_{1},e_{2} \in \Gamma(\mathcal{E})$ and let $f\in C^{\infty}(M)$. 
\begin{itemize}
\item [a)] $[e_{1},fe_{2}]= f[e_{1},e_{2}] +\tau(e_{1})[f](e_{2})$.
\item [b)] $\tau([e_{1},e_{2}]) = [\tau(e_{1}),\tau(e_{2})]_{Lie}$.
\item [c)] $[,]$ on $\Gamma(\mathcal{E})$ satisfies the Jacobi identity. 
\end{itemize}
In other words, $\tau:(\mathcal{E}, [,]) \longrightarrow (TM, [,]_{Lie})$ is a Lie algebroid structure on $\mathcal{E}$.
\end{prop}

It follows from Proposition (\ref{p:Liealgprop}) that we obtain a connection on $\mathfrak{g}$ defined by
\begin{align} \label{f:Liealgconn}
D_{v}\eta:= [v,\eta].
\end{align}

The algebroid bracket satisfies the following properties.

\begin{prop} \label{PROPBRA} Let $s:U \longrightarrow P$ be a local section of $P$, $x,y \in \mathfrak{g}$ and $v,w \in \Gamma(TM)$.
\begin{itemize}
\item[a)] $[\cdot,\cdot]$ restricts to a Lie bracket on each fiber of $\mathfrak{g}$.
\item[b)] $[\{s,x\}, \{s,y\}]=-\{s,[x,y]\}.$
\item[c)] $D_{v}\{s,x\} =\{s, [(s^{*}\bar{A})v, x]\}.$
\item[d)] $[v,w]= [v,w]_{Lie} - F(v,w)$.
\end{itemize}
\end{prop}

Furthermore, associated with the Lie algebroid $\tau:(\mathcal{E}, [\cdot,\cdot]) \longrightarrow (TM, [\cdot,\cdot]_{Lie})$ is the following exterior derivative which squares to zero.

\begin{defn} \label{DEFD1}
Define $d: \Gamma(\wedge^{k}\mathcal{E^{*}})\longrightarrow \Gamma(\wedge^{k+1}\mathcal{E^{*}})$ by 
\begin{align*}
d\sigma(e_{1},...,e_{k+1})=& \sum_{1\leq i\leq k+1}(-1)^{i-1}\tau(e_{i})[\sigma(e_{1},...,\hat{e_{i}},...e_{k+1})] +  \\  
 &  \sum_{1\leq i<j \leq k+1} (-1)^{i+j}\sigma([e_{i},e_{j}],e_{1},...,\hat{e_{i}},...,\hat{e_{j}},...,e_{k+1})].
\end{align*}
\end{defn}

\begin{prop}
For $\sigma \in \Gamma(\wedge^{k}\mathcal{E^{*}})$,
\[\widetilde{d\sigma}= d\tilde{\sigma}. \] 
\end{prop}

\begin{cor} Let $(\bar{g},\bar{J})$ be an invariant and compatible metric and complex structure on $P \rightarrow M$ and let $\bar{H}=-d^{c}\bar{w}$, where $\bar{w}(*,*)=\bar{g}(\bar{J}*,*)$. Also let $(g_{_{\mathcal{E}}},J,H)$ be the corresponding data on $\mathcal{E}$. Then $H(*,*,*)=dw(J*,J*,J*),$ where $w(*,*)=g_{_{\mathcal{E}}}(J*,*)$.
\end{cor}

In the upcoming sections we will need the following.
\begin{prop} \label{PROPDB1} Let $B \in \Gamma(\wedge^{2}\mathcal{E}^{*})$ and $e_{i}= \eta_{i}+v_{i} \in \Gamma(\mathcal{E}= \mathfrak{G} \oplus TM)$.
\begin{align*}
 dB(e_{1},e_{2},e_{3})= D_{v_{1}}B(e_{2},e_{3}) +B(F(v_{1},v_{2}),e_{3}) -B([\eta_{1},\eta_{2}],e_{3}) + cyclic(1,2,3).
\end{align*} 
\end{prop}

\begin{proof} In the following we will extend $D$ to a connection on $\mathcal{E}= \mathfrak{G} \oplus TM$ by setting $D=D \oplus \nabla^{L}$, where $\nabla^{L}$ is the Levi Civita connection associated with $g$. Using Definition \ref{DEFD1} and Proposition \ref{PROPBRA}, we have
\begin{align*}
dB(e_{1},e_{2},e_{3})&= v_{1}[B(e_{2},e_{3})] -B([e_{1},e_{2}],e_{3}) +cyclic(1,2,3)
                                  \\&=D_{v_{1}}B(e_{2},e_{3}) +B(D_{v_{1}}e_{2},e_{3}) +B(e_{2},D_{v_{1}}e_{3})
                                    \\ & \ \ \ -B([e_{1},e_{2}],e_{3})   +cyclic(1,2,3)    
       \\ &   =D_{v_{1}}B(e_{2},e_{3}) +B(\nabla^{L}_{v_{1}}v_{2},e_{3}) +B(e_{2}, \nabla^{L}_{v_{1}}v_{3})   
               \\ & \ \ \  +B(D_{v_{1}}\eta_{2},e_{3}) +B(e_{2},D_{v_{1}}\eta_{3})  
               -B([v_{1},v_{2}]_{Lie},e_{3})  +B(F(v_{1},v_{2}),e_{3})           
                  \\ & \ \ \    -B(D_{v_{1}}\eta_{2},e_{3}) +B(D_{v_{2}}\eta_{1},e_{3})
                      -B([\eta_{1},\eta_{2}],e_{3}) +cyclic(1,2,3)
        \\ &    = D_{v_{1}}B(e_{2},e_{3}) +B(F(v_{1},v_{2}),e_{3}) -B([\eta_{1},\eta_{2}],e_{3}) + cyclic(1,2,3).
    \end{align*}
\end{proof}

\subsection{Lie Algebroid Connection}
 Given the above structures on $\mathcal{E}$, we obtain a Lie algebroid connection $\nabla: \Gamma(\mathcal{E}) \longrightarrow \Gamma(\mathcal{E}^{*}\otimes \mathcal{E})$.  We define this and derive its action below.

\begin{defn}
Letting $e_{1}, e_{2}\in \Gamma(\mathcal{E})$, define $\nabla_{e_{1}} e_{2}$ to be the unique section of $\mathcal{E}$ that satisfies 
\[\widetilde{\nabla_{e_{1}} e_{2}}= \bar{\nabla}^{L}_{\tilde{e_{1}}}\tilde{e_{2}}, \]  where $\bar{\nabla}^{L}$ is the Levi Civita connection on $TP$ associated with $\bar{g}$.
\end{defn}

\begin{prop} \label{PROPNAB1} Let $e_{i} \in \Gamma(\mathcal{E})$ and $ f \in C^{\infty}(M)$.  Then
\begin{enumerate}
\item[1)] $\nabla_{e_{1}}e_{2} -  \nabla_{e_{2}}e_{1} = [e_{1},e_{2}]$.
\item[2)] $\nabla_{e_{1}}(fe_{2})=f \nabla_{e_{1}}e_{2} + \tau(e_{1})[f]e_{2}$.
\item[3)] $\tau(e_{1})[g_{_{\mathcal{E}}}(e_{2},e_{3})]= g_{_{\mathcal{E}}}(\nabla_{e_{1}}e_{2},e_{3}) + g_{_{\mathcal{E}}}(e_{2},\nabla_{e_{1}}e_{3})$.
\end{enumerate}
\end{prop}

Furthermore, we have the following Kozul formula for $\nabla$.
\begin{prop} \label{PROPNABKOZ} For $e_{i} \in \Gamma(\mathcal{E})$,
\begin{align*}
2g_{_{\mathcal{E}}}(\nabla_{e_{1}}e_{2}, e_{3})= \ & \tau(e_{1})[g_{_{\mathcal{E}}}(e_{2},e_{3})] +  \tau(e_{2})[g_{_{\mathcal{E}}}
(e_{1},e_{3})]-  \tau(e_{3})[g_{_{\mathcal{E}}}(e_{1},e_{2})] \\&+ g_{_{\mathcal{E}}}([e_{1},e_{2}],e_{3}) +g_{_{\mathcal{E}}}([e_{3},e_{1}],e_{2}) +g_{_{\mathcal{E}}}([e_{3},e_{2}],e_{1}).
\end{align*}
\end{prop}

We can then decompose $\nabla$ as follows:
\begin{prop} \label{PROPNABDEC} Let $\eta_{i} \in \Gamma(\mathfrak{g})$ and $v_{a} \in \Gamma(TM)$.  Then
\begin{enumerate}
\item[1)] $\nabla_{v_{1}}v_{2}=\nabla^{L}_{v_{1}}v_{2} -\tfrac{1}{2}F(v_{1},v_{2})$.
\item[2)] $\nabla_{v}\eta= D_{v}\eta +\tfrac{1}{2}G^{-1}D_{v}G(\eta, *) +\tfrac{1}{2}g^{-1}G(F(v,*), \eta)$.
\item[3)] $\nabla_{\eta}v= \tfrac{1}{2}G^{-1}D_{v}G(\eta,*)+\tfrac{1}{2}g^{-1}G(F(v,*),\eta)$.
\item[4)] $\nabla_{\eta_{1}}\eta_{2}= -\tfrac{1}{2}g^{-1}D_{*}G(\eta_{1},\eta_{2})+ \tfrac{1}{2}[\eta_{1},\eta_{2}] +\tfrac{1}{2}G^{-1}G([*,\eta_{1}],\eta_{2})+\tfrac{1}{2}G^{-1}G([*,\eta_{2}],\eta_{1})$.
\end{enumerate}
\end{prop}

\begin{proof} 
First we compute, using Proposition \ref{PROPNABKOZ},
\begin{align*}
 2g(\pi_{TM}\nabla_{v_{1}}v_{2},v_{3}) 
&= 2g_{_{\mathcal{E}}}(\nabla_{v_{1}}v_{2},v_{3}) \\&=  v_{1}[g(v_{2},v_{3})] + v_{2}[g(v_{1},v_{3})] - v_{3}[g(v_{1},v_{2})] 
 +g_{_{\mathcal{E}}}([v_{1},v_{2}],v_{3}) \\& \ \ \ + g_{_{\mathcal{E}}}([v_{3},v_{1}],v_{2}) + g_{_{\mathcal{E}}}([v_{3},v_{2}],v_{1}).
\end{align*}
Since $g_{_{\mathcal{E}}}([v_{i},v_{j}],v_{k})= g([v_{i},v_{j}]_{Lie},v_{k})$, the right hand side equals $2g(\nabla^{L}_{v_{1}}v_{2},v_{3})$. Hence $\pi_{TM}\nabla_{v_{1}}v_{2}= \nabla^{L}_{v_{1}}v_{2}.$  Furthermore, 
\begin{align*}
2G(\pi_{\mathfrak{g}}\nabla_{v_{1}}v_{2},\eta_{3})&= 2g_{_{\mathcal{E}}}(\nabla_{v_{1}}v_{2},\eta_{3}) \\&= v_{1}[g_{_{\mathcal{E}}}(v_{2},\eta_{3})] +v_{2}[g_{_{\mathcal{E}}}(v_{1},\eta_{3})] + \\& \ \ \ \ g_{_{\mathcal{E}}}([v_{1},v_{2}],\eta_{3}) + g_{_{\mathcal{E}}}([\eta_{3},v_{1}],v_{2})+ g_{_{\mathcal{E}}}([\eta_{3},v_{2}],v_{1}) \\&=  g_{_{\mathcal{E}}}([v_{1},v_{2}],\eta_{3}). 
\end{align*}
Using Proposition \ref{PROPBRA}, the right hand side equals $-G(F(v_{1},v_{2}),\eta_{3})$. Hence $\pi_{\mathfrak{g}}\nabla_{v_{1}}v_{2}=-\tfrac{1}{2}F(v_{1},v_{2})$, and part (1) follows.

For part (2), we first compute
\begin{align*}
g(\pi_{TM}\nabla_{v_{1}}\eta_{2},v_{3}) &= g_{_{\mathcal{E}}}(\nabla_{v_{1}}\eta_{2},v_{3}) \\& =-g_{_{\mathcal{E}}}(\eta_{2},\nabla_{v_{1}}v_{3})
                     								= \tfrac{1}{2}g_{_{\mathcal{E}}}(F(v_{1},v_{3}), \eta_{2}),						                   
\end{align*} 
which implies that $\pi_{TM}\nabla_{v_{1}}\eta_{2}= \tfrac{1}{2}g^{-1}G(F(v_{1},*),\eta_{2})$.   Furthermore
\begin{align*}
2G(\pi_{\mathfrak{g}}\nabla_{v_{1}}\eta_{2},\eta_{3})&= 2g_{_{\mathcal{E}}}(\nabla_{v_{1}}\eta_{2},\eta_{3}) \\ &= v_{1}[g_{_{\mathcal{E}}}(\eta_{2},\eta_{3})]+ g_{_{\mathcal{E}}}([v_{1},\eta_{2}],\eta_{3}) +g_{_{\mathcal{E}}}([\eta_{3},v_{1}],\eta_{2})
\\& =v_{1}[G(\eta_{2},\eta_{3})] + G(D_{v_{1}}\eta_{2},\eta_{3})-G(\eta_{2},D_{v_{1}}\eta_{3})
\\& =D_{v_{1}}G(\eta_{2},\eta_{3}) +2G(D_{v_{1}}\eta_{2},\eta_{3}).
\end{align*}
This implies $\pi_{\mathfrak{g}}\nabla_{v_{1}}\eta_{2} = \tfrac{1}{2}G^{-1}(D_{v_{1}}G)(\eta_{2},*) +D_{v_{1}}\eta_{2},$ finishing part (2).

For part (3), using Proposition \ref{PROPNAB1} we compute
\begin{align*}
\nabla_{\eta}v &=\nabla_{v}\eta -D_{v}\eta
              \\&= \tfrac{1}{2}G^{-1}(D_{v}G)(\eta,*) +\tfrac{1}{2} g^{-1}G(F(v,*), \eta),
\end{align*}
as required.  For part (4) we first compute
\begin{align*}
2g(\pi_{TM}\nabla_{\eta_{1}}\eta_{2},v_{3})&= 2g_{_{\mathcal{E}}}(\nabla_{\eta_{1}}\eta_{2},v_{3}) 
     \\&=-v_{3}[g_{_{\mathcal{E}}}(\eta_{1},\eta_{2})]+ g_{_{\mathcal{E}}}([v_{3},\eta_{1}],\eta_{2}) +g_{_{\mathcal{E}}}([v_{3},\eta_{2}],\eta_{1}) 
     \\&=-v_{3}[G(\eta_{1},\eta_{2})] +G(D_{v_{3}}\eta_{1},\eta_{2}) +G(D_{v_{3}}\eta_{2},\eta_{1})
     \\&=-DG_{v_{3}}(\eta_{1},\eta_{2}).
   \end{align*}  
This implies $\pi_{TM}\nabla_{\eta_{1}}\eta_{2}=-\tfrac{1}{2}g^{-1}DG_{*}(\eta_{1},\eta_{2})$.   Furthermore
\begin{align*}
2G(\pi_{\mathfrak{g}}\nabla_{\eta_{1}}\eta_{2},\eta_{3})
&= 2g_{_{\mathcal{E}}}(\nabla_{\eta_{1}}\eta_{2},\eta_{3}) 
\\& =g_{_{\mathcal{E}}}([\eta_{1},\eta_{2}],\eta_{3}) + g_{_{\mathcal{E}}}([\eta_{3},\eta_{1}],\eta_{2})+g_{_{\mathcal{E}}}([\eta_{3},\eta_{2}],\eta_{1}).
\\&= G([\eta_{1},\eta_{2}],\eta_{3}) + G([\eta_{3},\eta_{1}],\eta_{2})+G([\eta_{3},\eta_{2}],\eta_{1}).
\end{align*}
This shows that $\pi_{\mathfrak{g}}\nabla_{\eta_{1}}\eta_{2}=\tfrac{1}{2}[\eta_{1},\eta_{2}] +\tfrac{1}{2}G^{-1}G([*,\eta_{1}],\eta_{2}) +\tfrac{1}{2}G^{-1}G([*,\eta_{2}],\eta_{1}),$ finishing the proof of part (4).
\end{proof}

\subsection{Curvature and Torsion Tensors on Principal Bundles} \label{SECCURV}

Here we compute the components of the curvature tensors of the Lie algebroid connection $\nabla$. First, we define $R^{\nabla} \in \Gamma(\wedge^{2}\mathcal{E}^{*} \otimes End\mathcal{E})$ via
\begin{gather} \label{f:curvdef}
R^{\nabla}(e_{1},e_{2})e_{3}=\nabla_{e_{1}}\nabla_{e_{2}}e_{3} - \nabla_{e_{2}}\nabla_{e_{1}}e_{3}- \nabla_{[e_{1},e_{2}]}e_{3},
\end{gather}
for $e_{i} \in \Gamma(\mathcal{E})$.  Then define $R^{\nabla}(e_{1},e_{2},e_{3},e_{4})= g_{_{\mathcal{E}}}(R^{\nabla}(e_{1},e_{2})e_{3},e_{4})$. The Ricci and scalar curvatures of $\nabla$, $Ric^{\nabla}$ and $Scal^{\nabla}$, are defined using the appropriate traces over $g_{_{\mathcal{E}}}$.  Here and below we will adopt an abstract notation for taking traces of tensor quantities.  For instance, we set
\begin{align*}
D_{\cdot_{3}}G(\cdot_{1},\cdot_{2})D_{\cdot_{3}}G(\cdot_{1},\cdot_{2}) := \tr_{G}^{(1,2)}\tr_{g}^{3}D_{\cdot_{3}}G(\cdot_{1},\cdot_{2})D_{\cdot_{3}}G(\cdot_{1},\cdot_{2}) = G^{ij} G^{kl} g^{\ga \gb} D_{\ga} G_{ik} D_{\gb} G_{jl}.
\end{align*}
Note that we will always trace over $\cdot$'s while we will use *'s to denote the different components of tensors.  For clarity we will at times include the traces in the equations.

\begin{thm} \label{t:curvature} Let $\eta_{i} \in \Gamma(\mathfrak{g})$ and $v_{a} \in \Gamma(TM)$.
\begin{align*}
1)&\ R^{\nabla}(\eta_{1},\eta_{2},\eta_{3},\eta_{4})= -\tfrac{1}{4}D_{\cdot}G(\eta_{1},\eta_{4})D_{\cdot}G(\eta_{2},\eta_{3})
  -\tfrac{1}{4}G([\eta_{1},[\eta_{2},\eta_{3}]],\eta_{4}) \\ & -\tfrac{1}{4}G([\eta_{1},[\eta_{4},\eta_{3}]],\eta_{2})  -\tfrac{1}{4}G([\eta_{1},\eta_{3}],[\eta_{4},\eta_{2}])-\tfrac{1}{4}\tr_{G}G([\eta_{1},\cdot], \eta_{4})G([\eta_{2},\cdot], \eta_{3}) \\ & -\tfrac{1}{4}\tr_{G}G([\eta_{4},\cdot], \eta_{1})G([\eta_{2},\cdot], \eta_{3})  +(2 \leftrightarrow 3 \text{ on last four terms}) \\ & -(1 \leftrightarrow 2 \text{ on all terms})\\
2)&\ R^{\nabla}(\eta_{1},\eta_{2},v_{3},\eta_{4})= \tfrac{1}{4}D_{\cdot}G(\eta_{1},\eta_{4})G(F(v_{3},\cdot),\eta_{2})+
\tfrac{1}{4}\tr_{G}D_{v_{3}}G(\eta_{2},\cdot)G([\eta_{1},\cdot],\eta_{4}) \\ &+ \tfrac{1}{4}\tr_{G}D_{v_{3}}G(\eta_{2},\cdot)G([\eta_{4},\cdot],\eta_{1})  -\tfrac{1}{4}D_{v_{3}}G([\eta_{1},\eta_{2}],\eta_{4}) \\&  -\tfrac{1}{4}D_{v_{3}}G([\eta_{1},\eta_{4}],\eta_{2})-(1\leftrightarrow 2),\\
3)&\ R^{\nabla}(\eta_{1},v_{2},v_{3},\eta_{4})= -\tfrac{1}{2}(D_{v_{2}}DG)_{v_{3}}(\eta_{1},\eta_{4}) +\tfrac{1}{4}D_{v_{3}}G(\eta_{1},\cdot)D_{v_{2}}G(\eta_{4},\cdot)\\ &  +\tfrac{1}{4}G(F(v_{3},\cdot),\eta_{1})G(F(v_{2},\cdot),\eta_{4})-\tfrac{1}{4}G([\eta_{1},F(v_{2},v_{3})],\eta_{4}) \\& -\tfrac{1}{4}G([\eta_{4},F(v_{2},v_{3})],\eta_{1})  +\tfrac{1}{4}G([\eta_{1},\eta_{4}], F(v_{2},v_{3})),\\
4)&\ R^{\nabla}(\eta_{1},v_{2},v_{3},v_{4})= \tfrac{1}{4}D_{v_{4}}G(\eta_{1},F(v_{2},v_{3})) -\tfrac{1}{4}D_{v_{3}}G(\eta_{1},F(v_{2},v_{4}))-\tfrac{1}{2}D_{v_{2}}G(\eta_{1},F(v_{3},v_{4})) \\ & -\tfrac{1}{2}G(\eta_{1}, D_{v_{2}}F(v_{3},v_{4})),\\
5)&\ R^{\nabla}(v_{1},v_{2},v_{3},v_{4})= R^{\nabla^{L}}(v_{1},v_{2},v_{3},v_{4})+\tfrac{1}{2}G(F(v_{1},v_{2}), F(v_{3},v_{4}))-\tfrac{1}{4}G(F(v_{1},v_{4}), F(v_{2},v_{3})) \\ &+\tfrac{1}{4}G(F(v_{1},v_{3}), F(v_{2},v_{4})).
\end{align*}
\end{thm}

\begin{lemma} \label{l:nilpotent} Assuming $\mathcal{G}$ is nilpotent one has
\begin{align*}
0 =&\ \tr_{G}G([\eta, \cdot],\cdot),\\
0 =&\ \tr_{G}G([\eta_{1},[\eta_{2},\cdot]],\cdot),
\end{align*}
for all $\eta_{i} \in\mathfrak{G}$.
\begin{proof} Let $\eta_{i} \in\mathfrak{G}|_{m}$, for $m \in M$. Since $\mathcal{G}$ is nilpotent, $[\eta, *]$ and $[\eta_{1},[\eta_{2},*]] \in End(\mathfrak{G}|_{m})$ are traceless.
\end{proof}
\end{lemma}

\begin{prop} \label{p:Ricci} Assuming $\mathcal{G}$ is nilpotent one has
\begin{align*}
Ric^{\nabla}(\eta_{1},\eta_{2})=&\ -\tfrac{1}{2}(D_{\cdot}DG)_{\cdot}(\eta_{1},\eta_{2}) -\tfrac{1}{4}D_{\cdot_{2}}G(\cdot_{1},\cdot_{1})D_{\cdot_{2}}G(\eta_{1},\eta_{2})+\tfrac{1}{2}D_{\cdot_{2}}G(\eta_{1},\cdot_{1})D_{\cdot_{2}}G(\cdot_{1},\eta_{2}) \\ & +\tfrac{1}{4}G(F(\cdot_{1},\cdot_{2}), \eta_{1})G(F(\cdot_{1},\cdot_{2}), \eta_{2})-\tfrac{1}{2}\tr_{G}G([\cdot,\eta_{1}],[\cdot,\eta_{2}]) \\ &+\tfrac{1}{4}\tr_{G}G([\cdot_{1},\cdot_{2}],\eta_{1})G([\cdot_{1},\cdot_{2}],\eta_{2}),\\
Ric^{\nabla}(\eta,v)=&\ \tfrac{1}{2}G(\eta,D_{\cdot}F(v,\cdot)) +\tfrac{1}{2}D_{\cdot}G(\eta,F(v,\cdot))+ \tfrac{1}{4}G(\eta, F(v,\cdot_{2}))D_{\cdot_{2}}G(\cdot_{1}\cdot_{1}) \\& -\tfrac{1}{2}\tr_{G}D_{v}G([\cdot,\eta],\cdot),\\
Ric^{\nabla}(v_{1},v_{2}) =&\ Ric^{\nabla^{L}}(v_{1},v_{2}) -\tfrac{1}{2}(D_{v_{1}}DG)_{v_{2}}(\cdot,\cdot) +\tfrac{1}{4}D_{v_{1}}G(\cdot_{1},\cdot_{2})D_{v_{2}}G(\cdot_{1},\cdot_{2}) \\& -\tfrac{1}{2}G(F(v_{1},\cdot),F(v_{2},\cdot)).
\end{align*}
\end{prop}

\begin{prop} \label{p:scalar} Assuming $\mathcal{G}$ is nilpotent one has
\begin{align*}
R^{\N} =&\ R_g -(D_{\cdot_{1}}DG)_{\cdot_{1}}(\cdot_{2},\cdot_{2}) -\tfrac{1}{4}D_{\cdot_{3}}G(\cdot_{1},\cdot_{1})D_{\cdot_{3}}G(\cdot_{2},\cdot_{2}) \\&\ +\frac{3}{4}D_{\cdot_{3}}G(\cdot_{1},\cdot_{2})D_{\cdot_{3}}G(\cdot_{1},\cdot_{2}) -\tfrac{1}{4}G(F(\cdot_{1},\cdot_{2}), F(\cdot_{1},\cdot_{2})) -\tfrac{1}{4}\tr_{G}G([\cdot_{1},\cdot_{2}],[\cdot_{1},\cdot_{2}]).
\end{align*}
\end{prop}

We will now decompose the quantity $-d^{*}H$ which governs the evolution of $H$ in generalized Ricci flow.  For this we will extend $D$ to a connection on $\mathcal{E}= \mathfrak{G} \oplus TM$ by setting $D=D \oplus  \nabla^{L} $. 
\begin{prop} \label{p:dstarHdecomp} Assuming $\mathcal{G}$ is nilpotent one has
\begin{align*}
-d^{*}H =&\ D_{\cdot}H(\cdot, *,*) +\frac{1}{2}D_{\cdot_{1}}G(\cdot_{2},\cdot_{2})H(\cdot_{1}, *, * ) -D_{\cdot_{1}}G(\cdot_{2},\pi_{\mathfrak{G}}*) \wedge H(\cdot_{1},\cdot_{2},*)\\ &\ -\frac{1}{2}G(F(\cdot_{1},\cdot_{2}), \pi_{\mathfrak{G}}*) \wedge H(\cdot_{1},\cdot_{2},*) +\frac{1}{2}\tr_{G}G([\cdot_{1},\cdot_{2}],\pi_{\mathfrak{G}}*) \wedge H(\cdot_{1},\cdot_{2},*).
\end{align*}
\end{prop}

\begin{proof}In the following we will be using Proposition \ref{PROPNABDEC} without explicit reference. We will also let $\{\eta_{i}\}$ and $\{v_{a}\}$ be respective othonormal frames for $\mathfrak{G}$ and $TM$ and will set $e=\eta+w$ and $e'= \eta' +w' \in \mathcal{E}= \mathfrak{G} \oplus TM$.\\
First note: 
\begin{align*}
-d^{*}H(e,e')=\tr_{g_{_{\mathcal{E}}}}\nabla_{\cdot} H(\cdot,e,e')
           						=\nabla_{\eta_{i}}H(\eta_{i}, e,e') + \nabla_{v_{a}} H(v_{a},e,e').
\end{align*}    					
Consider then 
\begin{align*}
1) \ \nabla_{\eta_{i}}H(\eta_{i}, e,e') &= -H(\nabla_{\eta_{i}}\eta_{i}, e,e')-H(\eta_{i}, \nabla_{\eta_{i}}e,e')-H(\eta_{i},e,\nabla_{\eta_{i}}e')\\
                     &=\frac{1}{2}H( g^{-1}(D_{*}G)(\eta_{i},\eta_{i}), e, e') +\frac{1}{2}H(\eta_{i}, g^{-1}(D_{*}G)(\eta_{i},   	                   \eta), e')  \\& \hspace*{4mm} -\frac{1}{2}H(\eta_{i},[\eta_{i},\eta],e') -\frac{1}{2}H(\eta_{i},G^{-1}G([*,\eta_{i}],\eta),e')  \\& \hspace*{4mm}
                     -\frac{1}{2}H(\eta_{i},G^{-1}G([*,\eta],\eta_{i}),e') -\frac{1}{2}H(\eta_{i}, G^{-1}(D_{w}G)(\eta_{i},*),e')
                    \\& \hspace*{4mm}   -\frac{1}{2}H(\eta_{i}, g^{-1}G(F(w,*),\eta_{i}),e')  - (\text{primed} \leftrightarrow \text{unprimed} \text{ on last six terms}).
\end{align*}
Note that terms involving $\tr_{G}G([*,\cdot],\cdot)$ vanish by Lemma \ref{l:nilpotent} because $\mathcal{G}$ is nilpotent.
This then simplifies to 
\begin{align*}
\frac{1}{2}&D_{\cdot_{1}}G(\cdot_{2},\cdot_{2})H(\cdot_{1},e,e') -\frac{1}{2}D_{\cdot_{1}}G(\cdot_{2},\eta)H(\cdot_{1},\cdot_{2},e') -\frac{1}{2}H(F(w,\cdot),\cdot,e')  \\&  +\frac{1}{2}G([\cdot_{1},\cdot_{2}],\eta)H(\cdot_{1},\cdot_{2},e')
- (\text{primed} \leftrightarrow \text{unprimed} \text{ on last three terms}).
\end{align*}
Next, consider 
\begin{align*}
2) \ \nabla_{v_{a}} H(v_{a},e,e')&=v_{a}[H(v_{a},e,e')]- H(\nabla_{v_{a}}v_{a},e,e') -H(v_{a},\nabla_{v_{a}}e,e') -H(v_{a},e,\nabla_{v_{a}}e')\\
&=   D_{v_{a}}H(v_{a},e,e') +H(\nabla^{L}_{v_{a}}v_{a},e,e')+H(v_{a},D_{v_{a}}e,e') +H(v_{a},e,D_{v_{a}}e')  \\& \hspace*{4mm}  - H(\nabla_{v_{a}}v_{a},e,e')-H(v_{a},\nabla_{v_{a}}e,e') -H(v_{a},e,\nabla_{v_{a}}e')\\
  &= D_{v_{a}}H(v_{a},e,e') +\frac{1}{2}H(F(v_{a},v_{a}),e,e')   -\frac{1}{2}H(v_{a},G^{-1}(D_{v_{a}}G)(\eta,*), e')  \\& \hspace*{4mm}  -\frac{1}{2} H(v_{a},g^{-1}G(F(v_{a},*),\eta), e') +  \frac{1}{2}H(v_{a},F(v_{a},w),e') \\& \hspace*{4mm}  - (\text{primed} \leftrightarrow \text{unprimed} \text{ on last three terms}). \\
  &= D_{\cdot}H(\cdot,e,e') -\frac{1}{2}D_{\cdot_{1}}G(\eta,\cdot_{2})H(\cdot_{1},\cdot_{2},e') -\frac{1}{2}G(F(\cdot_{1},\cdot_{2}),\eta)H(\cdot_{1},\cdot_{2},e') \\& \hspace*{4mm} + \frac{1}{2}H(\cdot, F(\cdot,w),e') - (\text{primed} \leftrightarrow \text{unprimed} \text{ on last three terms}),
  \end{align*}					
  where, for $e= \eta +w$, we used \[D_{v}e- \nabla_{v}e= \frac{1}{2}F(v,w) -\frac{1}{2}G^{-1}(D_{v}G)(\eta,*) -\frac{1}{2}g^{-1}G(F(v,*),\eta).\]
	
	Combining the above terms yields the expression for  $-d^{*}H$ in the proposition.				                
\end{proof}

\section{Dimensional Reduction of Generalized Ricci Flow} \label{s:dimredGRF}

Let $\pi: P \rightarrow M$ be a nilpotent $\mathcal{G}$-principal bundle over a compact manifold equipped with an invariant generalized Ricci flow $(\bar{g}(t),\bar{H}(t))$. In this section we will first recast this solution  as a solution to a system of PDEs on $M$. We will then use the diffeomorphism action on the principal bundle to obtain gauge modified PDEs on $M$ which will be central in analyzing the functionals introduced in Section \ref{s:energy}.  

\subsection{Decomposition of families of metrics} \label{ss:families}

To decompose the Ricci flow equations according to the tensor decompositions of Section \ref{SECTDC}, we begin with a discussion of general one-parameter families.  Let now $\bar{g}(t)$ and $\bar{H}(t)$ be an invariant, time dependent metric and three form on a $\mathcal{G}$ principal bundle $\pi: P\rightarrow M$. Consider the  isomorphisms $\delta(t): TP \rightarrow \pi^{*}\mathcal{E}$ given by $\delta(Z)= \{p,\bar{A}Z \} + \pi_{*}Z,$ where $Z \in T_{p}P$ and $\bar{A}(t) $ are the  connections associated with $\bar{g}(t)$. Using this, as was done above, we obtain a time dependent fiberwise metric $g_{_{\mathcal{E}}}(t)= G(t) \oplus g(t)$ on $\mathcal{E}= \mathfrak{G}\oplus TM$, a section $H(t)$ of $\wedge^{3}\mathcal{E}^{*}$.  We further obtain a time dependent bracket $[,]$ and Lie algebroid connection $\nabla$ on $\mathcal{E}$ and the associated time dependent  operator $d$ and connection $D$ which were defined above for a fixed time. At times we will extend $D$ to a connection on $\mathcal{E}= \mathfrak{G} \oplus TM$ by setting $D=D \oplus  \nabla^{L}$, with associated curvature tensors.

It is important to note that since $\delta(t)$ is time dependent $\frac{\partial g_{_{\mathcal{E}}}}{\partial t}$ does not correspond to $\frac{\partial \bar{g}}{\partial t}$ but, by definition, $\dot{g}_{_{\mathcal{E}}}$ does. In general, if $\bar{B}(t)$ is a time dependent invariant tensor on $P$ then we define $\dot{B}(t)$ to be the tensor on $M$ that corresponds to $\frac{\partial  B}{\partial  t}$ using $\delta(t)$. 

To state some other notation, $\delta_{p}(t): T_{p}P \rightarrow T_{m}\mathcal{E}$, for $m=\pi(p)$, will denote the restriction of $\delta(t)$ at $p\in P$. It maps the vertical and horizontal tangent spaces $V_{p}P$ and $H^{t}_{p}P$ respectively to $\mathfrak{G}|_{m}$ and $T_{m}M$. Moreover, $\frac{dA}{dt}$ is the time dependent section of $T^{*}M \otimes \mathfrak{G}$ that is associated with $\frac{d\bar{A}}{dt}$. Similarly, $F(t) \in \wedge^{2}T^{*}M\otimes \mathfrak{G}$ corresponds to the curvature $\bar{F}(t):=d\bar{A}+\frac{1}{2}[\bar{A},\bar{A}]$.

\subsection{The Generalized Ricci Flow Equations}

Suppose $(\bar{g}(t),\bar{H}(t))$ is an invariant solution to the generalized Ricci flow equations on $P \rightarrow M$:
\begin{equation} \label{EQGRF2}
\frac{\partial \bar{g}}{\partial t} =\ -2Ric(\bar{g}) + \frac{1}{2} \bar{\mathcal{H}} \hspace{15mm}
\frac{\partial \bar{H}}{\partial t}  =\ \Delta_{\bar{g}} \bar{H},
\end{equation}
where $\bar{\mathcal{H}}(*,*)=\tr_{\bar{g}}\bar{H}(*,\cdot_{1},\cdot_{2})\bar{H}(*,\cdot_{1},\cdot_{2})$.   
To recast  $(\bar{g}(t),\bar{H}(t))$ as a solution of a system of PDEs on $M$, we first use the discussion of \S \ref{ss:families} to obtain $g_{_{\mathcal{E}}}(t)= G(t) \oplus g(t)$ and $H(t)$ on $\mathcal{E}= \mathfrak{G}\oplus TM$ along with the other time dependent data defined in that section.   In particular, using this equivalence, the equations in (\ref{EQGRF2}) are equivalent to:  
\begin{equation} \label{EQGRF3}
\dot{g}_{_{\mathcal{E}}}=-2Ric^{\nabla} +\frac{1}{2}\mathcal{H} \hspace{15mm} \dot{H}  =\ -dd^{*}H,
\end{equation}
where $\mathcal{H}(*,*)=\tr_{g_{_{\mathcal{E}}}}H(*,\cdot_{1},\cdot_{2})H(*,\cdot_{1},\cdot_{2})$, $d^{*}H=-\tr_{g_{_{\mathcal{E}}}}\nabla_{\cdot} H(\cdot,*,*)$ and where we used $d\bar{H}=0$. To decompose these equations we first prove a general decomposition formula for  $\dot{g}_{_{\mathcal{E}}}$. 

\begin{lemma} \label{l:generalvariation} Let $(\bar{g}_t, \bar{H}_t)$ be a one parameter family of invariant metrics and three-forms.  For $\eta_{i} \in \Gamma(\mathfrak{G})$ and $v_{a} \in \Gamma(TM)$, one has
\begin{itemize}
\item[1)] $\dot{g}_{_{\mathcal{E}}}(\eta_{1},\eta_{2})=\frac{\partial G}{\partial  t}(\eta_{1},\eta_{2})$.
\item[2)]  $\dot{g}_{_{\mathcal{E}}}(\eta,v)= G(\eta,\frac{dA}{dt}v)$.
\item[3)]   $\dot{g}_{_{\mathcal{E}}}(v_{1},v_{2})= \frac{\partial  g}{\partial t}(v_{1},v_{2})$.
\end{itemize}
\end{lemma}
\begin{proof} In the following, $\eta_{i} \in \mathfrak{G}|_{m}$, for $m\in M$ and $\tilde{\eta}_{i}= \delta_{p}^{-1}\eta_{i} \in V_{p}P$. Also, $v_{a} \in T_{m}M$ and $\tilde{v}_{a}(t)=\delta_{p}^{-1}(t)v_{a} \in H_{p}^{t}P$.  For Part 1) we compute
 \begin{equation*}
\dot{g}_{_{\mathcal{E}}}(\eta_{1},\eta_{2})= \frac{\partial\bar{g}}{\partial t}(\tilde{\eta}_{1},\tilde{\eta}_{2})=\frac{\partial}{\partial t}\bar{g}(\tilde{\eta}_{1},\tilde{\eta}_{2}) = \frac{\partial}{\partial t} g_{_{\mathcal{E}}}(\eta_{1},\eta_{2}) = \frac{\partial G}{\partial t}(\eta_{1},\eta_{2}), 
\end{equation*}
where we used $\frac{d\tilde{\eta}_{i}}{dt}=0$.  Next we compute
\begin{align*}
\dot{g}_{_{\mathcal{E}}}(\eta,v) = \frac{\partial \bar{g}}{\partial t}(\tilde{\eta}, \tilde{v})= -\bar{g}(\tilde{\eta}, \frac{d\tilde{v}}{dt})= \bar{g}(\tilde{\eta}, \lambda_{p}\frac{d\bar{A}}{dt}\tilde{v}) =G( \eta, \frac{dA}{dt}v), 
\end{align*}
where $\lambda_{p}: \mathfrak{g} \rightarrow V_{p}P$ is defined via $\bar{A}\lambda_{p}x=x$ for all $x \in \mathfrak{g}$.  Lastly we have
\begin{align*}
\dot{g}_{_{\mathcal{E}}}(v_{1},v_{2})= \frac{\partial \bar{g}}{\partial t}(\tilde{v}_{1},\tilde{v}_{2}) = \frac{\partial}{\partial t}\bar{g}(\tilde{v}_{1},\tilde{v}_{2}) =  \frac{\partial}{\partial t}g_{_{\mathcal{E}}}(v_{1},v_{2})= \frac{\partial g}{\partial t}(v_{1},v_{2}),
\end{align*}
where we used that $\bar{g}(\frac{d\tilde{v}_{i}}{dt},\tilde{v}_{j})=0.$
\end{proof}

Using Lemma \ref{l:generalvariation} together with Proposition \ref{p:Ricci} gives the following decomposition of the metric evolution equation in (\ref{EQGRF3}).

\begin{prop} \label{p:decGRF} Given $(\bar{g}(t), \bar{H}(t))$ an invariant solution of generalized Ricci flow as above, one has
\begin{align*}
1)&\ \frac{\partial G}{\partial t}(\eta_{1},\eta_{2}) = (D_{\cdot}DG)_{\cdot}(\eta_{1},\eta_{2}) +\frac{1}{2}D_{\cdot_{2}}G(\cdot_{1},\cdot_{1})D_{\cdot_{2}}G(\eta_{1},\eta_{2})-D_{\cdot_{2}}G(\eta_{1},\cdot_{1})D_{\cdot_{2}}G(\cdot_{1},\eta_{2}) \\ &\ \qquad \qquad \qquad -\frac{1}{2}G(F(\cdot_{1},\cdot_{2}), \eta_{1})G(F(\cdot_{1},\cdot_{2}), \eta_{2}) + \tr_{G}G([\cdot,\eta_{1}],[\cdot,\eta_{2}]) \\ &\ \qquad \qquad \qquad -\frac{1}{2}\tr_{G}G([\cdot_{1},\cdot_{2}],\eta_{1})G([\cdot_{1},\cdot_{2}],\eta_{2}) + \frac{1}{2}\mathcal{H}(\eta_{1},\eta_{2}),\\
2)&\ G(\frac{dA}{dt}v, \eta) = -G(D_{\cdot}F(v,\cdot),\eta) -D_{\cdot}G(F(v,\cdot),\eta)- \frac{1}{2}G(F(v,\cdot_{2}),\eta)D_{\cdot_{2}}G(\cdot_{1},\cdot_{1}) \\&\ \qquad \qquad \qquad +\tr_{G}D_{v}G([\cdot,\eta],\cdot) + \frac{1}{2}\mathcal{H}(v,\eta),\\
3)&\ \frac{\partial g}{ \partial t}(v_{1},v_{2}) = -2Ric^{\nabla^{L}}(v_{1},v_{2}) +(D_{v_{1}}DG)_{v_{2}}(\cdot,\cdot) -\frac{1}{2}D_{v_{1}}G(\cdot_{1},\cdot_{2})D_{v_{2}}G(\cdot_{1},\cdot_{2}) \\&\ \qquad \qquad \qquad +G(F(v_{1},\cdot),F(v_{2},\cdot))+ \frac{1}{2}\mathcal{H}(v_{1},v_{2}).
\end{align*}
\end{prop}

\subsection{Gauge modified generalized Ricci flow}

Suppose $(\bar{g}(t),\bar{H}(t))$ is a solution to the generalized Ricci flow equations (\ref{EQGRF2}) on $P \rightarrow M$.  We will modify these equations by using flows of time dependent vector fields on $P$ and will then derive the corresponding system of PDEs on the base manifold.  We will be using the notation of Section \ref{SECTDC} throughout.

\subsubsection{Invariant generalized Ricci flow in the canonical gauge} \label{ss:CGF}

\begin{defn} \label{d:qdef} Given $P \to M$ a principal bundle and $\bar{g}$ an invariant metric, let
\begin{align*}
q :=&\ - \tfrac{1}{2} g^{-1} D G(\cdot, \cdot) \in TM\\
X :=&\ \gd^{-1} \pi^* q \in TP,
\end{align*}
where $\gd^{-1} \pi^*$ denotes the horizontal lift with respect to the connection $\bar{A}$ associated with $\bar{g}$.
\end{defn}

Now, given $(\bar{g}(t), \bar{H}(t))$ a solution of generalized Ricci flow, we obtain associated one parameter families of vector fields $q(t)$, $X(t)$ as in Definition \ref{d:qdef}.  Let $\phi_{t}$ and $\psi_{t}$ be the respective flows on $M$ and $P$ such that $\phi_{t=0}(m)=m$ and $\psi_{t=0}(p)=p$ for all $m \in M$ and $p \in P$.  Now define $\hat{g}(t)=\psi^{*}\bar{g}(t)$ and $\hat{H}(t)=\psi^{*}\bar{H}(t)$, which are time dependent invariant metrics and three forms on $P$. By a standard computation they satisfy the following modified generalized Ricci flow equations, which we will refer to as generalized Ricci flow \emph{in the canonical gauge}: 
 \begin{equation} \label{EQGRF3}
\frac{\partial \hat{g}}{\partial t} =\ -2Ric(\hat{g}) + \frac{1}{2} \hat{\mathcal{H}} + \mathcal{L}_{(\psi^{-1})_{*}X}\hat{g} \hspace{15mm}
\frac{\partial \hat{H}}{\partial t}  =\ \Delta_{\hat{g}} \hat{H} + d i_{(\psi^{-1})_{*}X}\hat{H},
\end{equation}
where  $\hat{\mathcal{H}}(*,*)=\tr_{\hat{g}}\hat{H}(*,\cdot_{1},\cdot_{2})\hat{H}(*,\cdot_{1},\cdot_{2})$.

Given $(\hat{g}(t), \hat{H}(t))$, we have the following data that corresponds to that defined in Section \ref{SECTDC}: 
the time dependent isomorphism $\hat{\delta(t)}: TP \rightarrow \pi^{*}\mathcal{E}$ given by $\hat{\delta}(Z)= \{p,\hat{A}Z \} + \pi_{*}Z,$ where $Z \in T_{p}P$ and $\hat{A}(t) $ are the connections associated with $\hat{g}(t)$; a time dependent fiberwise metric $g'_{_{\mathcal{E}}}(t)= G'(t) \oplus g'(t)$ on $\mathcal{E}= \mathfrak{G}\oplus TM$ and a section $H'(t)$ of $\wedge^{3}\mathcal{E}^{*}$.  We further obtain a time dependent bracket $[,]'$ and Lie algebroid connection $\nabla'$ on $\mathcal{E}$ and the associated time dependent  operator $d'$ and connection $D'$. In addition, $\frac{dA'}{dt}$ is the time dependent section of $T^{*}M \otimes \mathfrak{G}$ that is associated with  $\frac{d\hat{A}}{dt}$ and $F'(t) \in \wedge^{2}T^{*}M\otimes \mathfrak{G}$ corresponds to the curvature $\hat{F}(t):=d\hat{A}+\frac{1}{2}[\hat{A},\hat{A}]$.

To reformulate Equations \ref{EQGRF3} in terms of this data first set $q'=-\frac{1}{2}g'^{-1}D'G'(\cdot,\cdot)$ and note that $(\psi^{-1})_{*}X(t)= \hat{\delta}(t)^{-1}\pi^{*}q'(t),$ which is the horizontal lift of $q'(t)$ using the connection $\hat{A}(t)$. Equations \ref{EQGRF3} are then equivalent to:
\begin{equation} \label{f:GFF}
\dot{g'_{_{\mathcal{E}}}}=-2Ric^{\nabla'} +\frac{1}{2}\mathcal{H'} +\mathcal{L}_{q'} g'_{_{\mathcal{E}}} \hspace{15mm} \dot{H'}  =\ -d'd'^{*}H'+ d'i_{q'}H',
\end{equation}
where $\mathcal{H'}(*,*)=\tr_{g'_{_{\mathcal{E}}}}H'(*,\cdot_{1},\cdot_{2})H'(*,\cdot_{1},\cdot_{2})$, $\mathcal{L}_{q'} g'_{_{\mathcal{E}}}(e_{1},e_{2})= g'_{_{\mathcal{E}}}(\nabla'_{e_{1}}q',e_{2}) + g'_{_{\mathcal{E}}}(e_{1}, \nabla'_{e_{2}}q')$ for $e_{i} \in \mathcal{E}$, $d'^{*}H'=-\tr_{g'_{_{\mathcal{E}}}}\nabla'_{\cdot} H'(\cdot,*,*)$ and where we used $d\hat{H}=0$.  We can then decompose the equation for $\dot{g}'_{_{\mathcal{E}}}$, analogously to Proposition \ref{p:decGRF}.  We need a preliminary lemma.

\begin{lemma} \label{LEMMADDG1}Let $v_{i} \in T_{m}M$. Then
\begin{align*}
1)&\ \nabla^{L}_{v}q = -\frac{1}{2}g^{-1}(D_{v}DG)_{*}(\cdot,\cdot) + \frac{1}{2}g^{-1}(D_{v}G(\cdot_{1},\cdot_{2})D_{*}G(\cdot_{1},\cdot_{2}))\\
2)&\ (D_{v_{1}}DG)_{v_{2}}(\eta_{1},\eta_{2})=(D_{v_{2}}DG)_{v_{1}}(\eta_{1},\eta_{2}) +G([F(v_{1},v_{2}),\eta_{1}],\eta_{2})  +G(\eta_{1},[F(v_{1},v_{2}),\eta_{2}]),\\
3)&\ (D_{v_{1}}DG)_{v_{2}}(\cdot,\cdot)= (D_{v_{2}}DG)_{v_{1}}(\cdot,\cdot).
\end{align*}
\end{lemma}

\begin{proof} To prove Part 1), first let $e_{i}$ be an orthonormal frame for $\mathcal{E}$ and consider  
\begin{align*}
\nabla^{L}_{v} (DG(e_{i},e_{i}))[w]&= D_{v}(D_{w}G(e_{i},e_{i}))-D_{\nabla^{L}_{v}w}G(e_{i},e_{i})
                   \\ &= (D_{v}DG)_{w}(e_{i},e_{i}) +2D_{w}G(D_{v}e_{i},e_{i})
                   \\ & = (D_{v}DG)_{w}(\cdot,\cdot)- D_{v}G(\cdot_{1},\cdot_{2})D_{w}G(\cdot_{1},\cdot_{2}),
\end{align*}
where we used $D_{w}G(D_{v}e_{i},e_{i}) =-\frac{1}{2}D_{v}G(\cdot_{1},\cdot_{2})D_{w}G(\cdot_{1},\cdot_{2}).$
Part 1) then follows from the relation $\nabla^{L}_{v}q= -\frac{1}{2}g^{-1}\nabla^{L}_{v} (DG(\cdot,\cdot))$. 

For Part 2), we first compute
\begin{align*}
R^{D}(v_{1},v_{2})\eta &= D_{v_{1}}D_{v_{2}}\eta -D_{v_{2}}D_{v_{1}}\eta-D_{[v_{1},v_{2}]_{Lie}}\eta
                       \\ &=[v_{1},[v_{2},\eta]] -[v_{2},[v_{1},\eta]] -[[v_{1},v_{2}]_{Lie},\eta]
                      \\ &  = [v_{1},[v_{2},\eta]] +[v_{2},[\eta,v_{1}]] + [\eta, [v_{1},v_{2}]] -[F(v_{1},v_{2}),\eta]
                       \\ & = -[F(v_{1},v_{2}),\eta],                      
\end{align*}
where we used Proposition \ref{PROPBRA} d) and Proposition \ref{p:Liealgprop} c).

Next consider  the curvature acting on $G$:
\begin{align*} 
-R^{D}(v_{1},v_{2})\cdot G(\eta_{1},\eta_{2})&:= (D_{v_{1}}D_{v_{2}}G -D_{v_{2}}D_{v_{1}}G-D_{[v_{1},v_{2}]_{Lie}}G)(\eta_{1},\eta_{2}) \\&
              =-G(R^{D}(v_{1},v_{2})\eta_{1},\eta_{2}) -G(\eta_{1}, R^{D}(v_{1},v_{2}),\eta_{2})
            \\ &  =G([F(v_{1},v_{2}),\eta_{1}],\eta_{2}) + G(\eta_{1},[F(v_{1},v_{2}),\eta_{2}]).
\end{align*}
Part 2) then follows from the relation $-R^{D}(v_{1},v_{2})\cdot G(\eta_{1},\eta_{2})= (D_{v_{1}}DG)_{v_{2}}(\eta_{1},\eta_{2})-(D_{v_{2}}DG)_{v_{1}}(\eta_{1},\eta_{2})$.

Part 3) then follows from Part 2) and the assumption that $\mathcal{G}$ is nilpotent, which implies  $G([\eta,\cdot],\cdot)=0,$ for $\eta \in \mathfrak{G}$ by Lemma \ref{l:nilpotent}.
\end{proof}

\begin{prop} \label{p:decgfGRF1} Given $(\bar{g}(t), \bar{H}(t))$ an invariant solution of gauge-fixed generalized Ricci flow (\ref{EQGRF3}), one has
\begin{align*}
1)&\ \frac{\partial G'}{\partial t}(\eta_{1},\eta_{2}) = -2Ric^{\nabla'}(\eta_{1},\eta_{2}) -\frac{1}{2}D'_{\cdot_{2}}G'(\cdot_{1},\cdot_{1})D'_{\cdot_{2}}G'(\eta_{1},\eta_{2}) + \frac{1}{2}\mathcal{H}'(\eta_{1},\eta_{2}),\\
2)&\ G'\left(\frac{dA'}{dt}v,\eta \right) = -2Ric^{\nabla'}(v,\eta) +\frac{1}{2}G'(F'(v,\cdot_{2}),\eta)D'_{\cdot_{2}}G'(\cdot_{1},\cdot_{1}) + \frac{1}{2}\mathcal{H}'(v,\eta),\\
3)&\ \frac{\partial g'}{ \partial t}(v_{1},v_{2}) = -2Ric^{\nabla'}(v_{1},v_{2})  -(D'_{v_{1}}D'G')_{v_{2}}(\cdot,\cdot)+ D'_{v_{1}}G'(\cdot_{1},\cdot_{2})D'_{v_{2}}G'(\cdot_{1},\cdot_{2}) + \frac{1}{2}\mathcal{H}'(v_{1},v_{2}).
\end{align*}
\end{prop}

\begin{proof} We will first compute the different components of $\mathcal{L}_{q'} g'_{_{\mathcal{E}}}(e_{1},e_{2})= g'_{_{\mathcal{E}}}(\nabla'_{e_{1}}q',e_{2}) + g'_{_{\mathcal{E}}}(e_{1}, \nabla'_{e_{2}}q')$. We will be using Proposition \ref{PROPNABDEC} throughout. Moreover, to simplify the notation we will be removing all the primes from the equations.  Let $\eta_{i} \in \mathfrak{G}|_{m}$ and $v_{a} \in T_{m}M$.  First we compute
\begin{align*}
 1) \mathcal{L}_{q} g_{_{\mathcal{E}}}(\eta_{1},\eta_{2}) 
&= G(\pi_{\mathfrak{G}}\nabla_{\eta_{1}}q,\eta_{2}) + (1\leftrightarrow 2)\\
&= \frac{1}{2}G(G^{-1}(D_{q}G)(\eta_{1},*),\eta_{2}) + (1\leftrightarrow 2)\\
&= D_{q}G(\eta_{1},\eta_{2})\\
&= -\frac{1}{2}D_{\cdot_{2}}G(\cdot_{1},\cdot_{1})D_{\cdot_{2}}G(\eta_{1},\eta_{2}).
\end{align*}
Next we compute
\begin{align*}
2) \mathcal{L}_{q} g_{_{\mathcal{E}}}(v,\eta)= &
             g_{_{\mathcal{E}}}(\nabla_{v}q,\eta) + g_{_{\mathcal{E}}}(v,\nabla_{\eta}q)
           \\ &  =-\frac{1}{2}G(F(v,q),\eta)+ \frac{1}{2} g(v, g^{-1}G(F(q,*),\eta))
            \\ &  = G(F(q,v),\eta)
            \\ &= \frac{1}{2}D_{\cdot_{2}}G(\cdot_{1},\cdot_{1})G(F(v,\cdot_{2}),\eta).
\end{align*}
3) Using Lemma \ref{LEMMADDG1}, we have
\begin{align*}
 \mathcal{L}_{q} g_{_{\mathcal{E}}}(v_{1},v_{2}) 
   &= g_{_{\mathcal{E}}}(\nabla_{v_{1}}q,v_{2}) + (1\leftrightarrow 2)\\
   &= g(\nabla^{L}_{v_{1}}q,v_{2}) + (1\leftrightarrow 2)\\
   &=-\frac{1}{2}g(g^{-1}(D_{v_{1}}DG)_{*}(\cdot,\cdot),v_{2}) + \frac{1}{2}g(g^{-1}(D_{v_{1}}G(\cdot_{1},\cdot_{2})D_{*}G(\cdot_{1},\cdot_{2})),v_{2})  + (1\leftrightarrow 2)\\
   &= -\frac{1}{2}(D_{v_{1}}DG)_{v_{2}}(\cdot,\cdot) + \frac{1}{2}D_{v_{1}}G(\cdot_{1},\cdot_{2})D_{v_{2}}G(\cdot_{1},\cdot_{2}) + (1\leftrightarrow 2)\\
   &= -(D_{v_{1}}DG)_{v_{2}}(\cdot,\cdot) + D_{v_{1}}G(\cdot_{1},\cdot_{2})D_{v_{2}}G(\cdot_{1},\cdot_{2})
\end{align*}
The result follows by combining these results with Proposition \ref{l:generalvariation} and Equation \ref{f:GFF}.
\end{proof}

Combining Proposition \ref{p:decgfGRF1} with Theorem \ref{t:curvature} and removing the primes yields the following modified generalized Ricci flow equations:

\begin{prop} \label{p:decgfGRF2} Given $(\bar{g}_t, \bar{H}_t)$ a solution of canonically gauge-fixed generalized Ricci flow (\ref{EQGRF3}), one has
\begin{align*}
1)&\  \frac{\partial G}{ \partial t}(\eta_{1},\eta_{2})= (D_{\cdot}DG)_{\cdot}(\eta_{1},\eta_{2}) -D_{\cdot_{2}}G(\eta_{1},\cdot_{1})D_{\cdot_{2}}G(\cdot_{1},\eta_{2}) \\ &\qquad \qquad \qquad -\frac{1}{2}G(F(\cdot_{1},\cdot_{2}), \eta_{1})G(F(\cdot_{1},\cdot_{2}), \eta_{2}) + \tr_{G}G([\cdot,\eta_{1}],[\cdot,\eta_{2}]) \\ &\qquad \qquad \qquad -\frac{1}{2}\tr_{G}G([\cdot_{1},\cdot_{2}],\eta_{1})G([\cdot_{1},\cdot_{2}],\eta_{2}) + \frac{1}{2}\mathcal{H}(\eta_{1},\eta_{2}),\\
 2)&\ G(\frac{dA}{dt}v, \eta) = -G(D_{\cdot}F(v,\cdot), \eta) -D_{\cdot}G(F(v,\cdot),\eta) +\tr_{G}D_{v}G([\cdot,\eta],\cdot) + \frac{1}{2}\mathcal{H}(v,\eta),\\
3)&\ \frac{\partial g}{ \partial t}(v_{1},v_{2}) = -2Ric^{\nabla^{L}}(v_{1},v_{2}) +\frac{1}{2}D_{v_{1}}G(\cdot_{1},\cdot_{2})D_{v_{2}}G(\cdot_{1},\cdot_{2})\\
&\ \qquad \qquad \qquad +G(F(v_{1},\cdot),F(v_{2},\cdot))+ \frac{1}{2}\mathcal{H}(v_{1},v_{2}),\\
4)&\ \dot{H}= d\dot{B}, \mbox{ where}\\
& \qquad \dot{B} = D_{\cdot}H(\cdot, *,*)  -D_{\cdot_{1}}G(\cdot_{2},\pi_{\mathfrak{G}}*) \wedge H(\cdot_{1},\cdot_{2},*)  \\ 
& \qquad \qquad  -\frac{1}{2}G(F(\cdot_{1},\cdot_{2}), \pi_{\mathfrak{G}}*) \wedge H(\cdot_{1},\cdot_{2},*) +\frac{1}{2}\tr_{G}G([\cdot_{1},\cdot_{2}],\pi_{\mathfrak{G}}*) \wedge H(\cdot_{1},\cdot_{2},*).
\end{align*}
\end{prop}
\begin{proof} The proofs of Parts 1) -- 3) were already given. The proof of Part 4) follows from Equation \ref{f:GFF}, Proposition \ref{p:dstarHdecomp} and the relation $H(q,*,*)=-\frac{1}{2}D_{\cdot_{1}}G(\cdot_{2},\cdot_{2})H(\cdot_{1},*,*)$.
\end{proof}

\subsubsection{Invariant generalized Ricci flow in a general gauge}

In addition to the canonical gauge fixing of the previous section, we will need to consider general gauge modifications by gradient vector fields.  This is key to analyzing the functionals in Section \ref{s:energy}.  In particular, suppose $(\bar{g}(t),\bar{H}(t))$ is a solution of the generalized Ricci flow. Given $f(t) \in C^{\infty}(M)$ a one parameter family of smooth functions, let $q_f(t)=-\frac{1}{2}g^{-1}DG(\cdot,\cdot) -\nabla f(t)$, $X_f(t) = \gd(t)^{-1} \pi^* q_f(t)$.  Pulling back by the flow generated by $X_f(t)$, as in Section \ref{ss:CGF}, we obtain the following modified generalized Ricci flow equations:
 \begin{equation} \label{f:gengfGRF}
\frac{\partial \bar{g}}{\partial t} =\ -2Ric(\bar{g}) + \frac{1}{2} \bar{\mathcal{H}} + \mathcal{L}_{X_f}\bar{g} \hspace{15mm}
\frac{\partial \bar{H}}{\partial t}  =\ \Delta_{\bar{g}} \bar{H} + di_{X_f}\bar{H},
\end{equation}
An analysis similar to that of Propositions \ref{p:decgfGRF1} and \ref{p:decgfGRF2} yields the relevant decomposed evolution equations.

\begin{prop} \label{p:decgengfGRF} Given $f_t$ a one parameter family of smooth functions and $(\bar{g}_t, \bar{H}_t)$ a solution of (\ref{f:gengfGRF}), one has
\begin{align*}
1)&\ \frac{\partial G}{ \partial t}(\eta_{1},\eta_{2}) = (D_{\cdot}DG)_{\cdot}(\eta_{1},\eta_{2}) -D_{\cdot_{2}}G(\eta_{1},\cdot_{1})D_{\cdot_{2}}G(\cdot_{1},\eta_{2}) \\ 
&\ \qquad \qquad \qquad -\frac{1}{2}G(F(\cdot_{1},\cdot_{2}), \eta_{1})G(F(\cdot_{1},\cdot_{2}), \eta_{2}) + \tr_{G}G([\cdot,\eta_{1}],[\cdot,\eta_{2}]) \\ 
&\ \qquad \qquad \qquad -\frac{1}{2}\tr_{G}G([\cdot_{1},\cdot_{2}],\eta_{1})G([\cdot_{1},\cdot_{2}],\eta_{2}) + \frac{1}{2}\mathcal{H}(\eta_{1},\eta_{2}) -D_{\nabla f}G(\eta_{1},\eta_{2}),\\
2)&\ G(\frac{dA}{dt}v, \eta) = -G(D_{\cdot}F(v,\cdot), \eta) -D_{\cdot}G(F(v,\cdot),\eta) +\tr_{G}D_{v}G([\cdot,\eta],\cdot)\\
&\ \qquad \qquad \qquad + \frac{1}{2}\mathcal{H}(v,\eta) -G(F(\nabla f, v), \eta),\\
3)&\ \frac{\partial g}{ \partial t}(v_{1},v_{2}) = -2Ric^{\nabla^{L}}(v_{1},v_{2}) +\frac{1}{2}D_{v_{1}}G(\cdot_{1},\cdot_{2})D_{v_{2}}G(\cdot_{1},\cdot_{2}) +G(F(v_{1},\cdot),F(v_{2},\cdot)) \\ 
&\ \qquad \qquad \qquad + \frac{1}{2}\mathcal{H}(v_{1},v_{2}) -2\N^2 f(v_{1},v_{2}),\\
4)&\ \dot{H}= d\dot{B}, \mbox{ where }\\
&\ \qquad \dot{B} =  D_{\cdot}H(\cdot, *,*)  -D_{\cdot_{1}}G(\cdot_{2},\pi_{\mathfrak{G}}*) \wedge H(\cdot_{1},\cdot_{2},*)  -\frac{1}{2}G(F(\cdot_{1},\cdot_{2}), \pi_{\mathfrak{G}}*) \wedge H(\cdot_{1},\cdot_{2},*) \\ 
&\ \qquad \qquad +\frac{1}{2}\tr_{G}G([\cdot_{1},\cdot_{2}],\pi_{\mathfrak{G}}*) \wedge H(\cdot_{1},\cdot_{2},*) -H(\nabla f, *,*).
\end{align*}
\end{prop} 

\section{Energy and entropy functionals} \label{s:energy}
\subsection{Conjugate heat equation}

A key component of Perelman's energy and entropy monotonicity is the coupling of Ricci flow to a solution of its associated conjugate heat equation, which in the relevant gauge fixes a background measure on the given manifold.  In our setting we are forced to work on the base space of a given principal bundle, and so the relevant conjugate heat equation should be designed to fix a background measure on this base space.  This is partly responsible for the fact that the evolution equations and monotonicity formulas to follow are not direct consequences of the existing formal monotonicity formulas on the total space of the bundle.

\begin{defn} \label{d:conjugateheat} Given $(\bar{g}_t, \bar{H}_t)$ a solution to generalized Ricci flow, we say that $u_t \in C^{\infty}(M)$ is a solution of the \emph{conjugate heat equation} if
\begin{align} \label{f:conjugateheat}
\dt u =&\ - \gD u + \left(R_g - \tfrac{1}{4} \brs{D G}^2 - \tfrac{1}{2} \brs{F}^2 - \tfrac{1}{4} \tr_g \mathcal H + \tfrac{1}{2} \IP{q, \N \log u} \right)u.
\end{align}
\end{defn}

More precisely, this equation arises as the heat equation conjugate to the generalized Ricci flow as expressed in the canonical gauge, then pulled back to the ungauged generalized Ricci flow.  We note that in the case of the steady entropy functional $\FF$ we obtain the potential function $f$ by $u = e^{-f}$, while for the  expander entropy one has $u = \frac{e^{-f}}{(4 \pi t)^{\frac{n}{2}}}$.

\subsection{Energy functional}

\begin{defn} \label{d:energy} Given $\bar{g}, \bar{H}$ as above and $f \in C^{\infty}(M)$, let
\begin{gather}
\begin{split}
\mathcal{F}(\bar{g},\bar{H},f) =&\ \int_{M} \left\{ |\N f|^{2}+R_{g} -\tfrac{1}{4} |DG|^{2} -\tfrac{1}{4} |F|^{2} -\frac{1}{12} |\bar{H}|_{\bar{g}}^{2} -\tfrac{1}{4} \brs{[,]}^2 \right\} e^{-f} dV_g.
\end{split}
\end{gather}
\end{defn}

We now proceed to compute the evolution of $\FF$ along a solution to generalized Ricci flow.  This will require a buildup of general variational lemmas.

\begin{lemma} \label{l:variation1} 
 Let $\eta_{i} \in \Gamma(\mathfrak{G})$ and $v_{a}\in \Gamma(TM)$.
\begin{align*}
&1) \frac{d}{dt}(D_{v}\eta)= -[\frac{dA}{dt}v,\eta]
\\&2) \frac{dDG}{dt}(v,\eta_{1},\eta_{2})= (D_{v}\frac{dG}{dt})(\eta_{1},\eta_{2}) + G([\frac{dA}{dt}v,\eta_{1}],\eta_{2}) + G(\eta_{1}, [\frac{dA}{dt}v,\eta_{2}]).
\\&3) \frac{dF}{dt}(v_{1},v_{2})= (d^{D}\frac{dA}{dt})(v_{1},v_{2}).
\end{align*}

\end{lemma}

\begin{proof}
1) It suffices to prove part 1) when $\eta=\{s,x\} \in \Gamma(\mathfrak{G})$, where $s:U \rightarrow P $ is a local section and $x \in \mathfrak{g}$. For this let $p= s(m)$, for $m \in M$ and consider 
\begin{align*}
\frac{d}{dt}D_{v}\{s,x\}|_{m}&= \frac{d}{dt}\{s, [(s^{*}\bar{A})v,x] \} |_{m}
                   \\&=\{p, [\frac{d\bar{A}}{dt}s_{*}v|_{m},x] \} = -[\{p,\frac{d\bar{A}}{dt}s_{*}v|_{m} \}, \{p,x \}]
                   \\&=-[\frac{dA}{dt}v,\eta]|_{m},
\end{align*}
as claimed.  For part 2) we directly compute
\begin{align*}
 \frac{dDG}{dt}(v,\eta_{1},\eta_{2})=& \frac{d}{dt}(v[G(\eta_{1},\eta_{2})] -G(D_{v}\eta_{1}, \eta_{2}) -G(\eta_{1},D_{v}\eta_{2}))
\\&=v[\frac{dG}{dt}(\eta_{1},\eta_{2})] - \frac{dG}{dt}(D_{v}\eta_{1}, \eta_{2 }) -\frac{dG}{dt}(\eta_{1},D_{v}\eta_{2})\\& \ \ \  -G(\frac{d}{dt}D_{v}\eta_{1},\eta_{2}) -G(\eta_{1},\frac{d}{dt}D_{v}\eta_{2})
\\& =(D_{v}\frac{dG}{dt})(\eta_{1},\eta_{2}) + G([\frac{dA}{dt}v,\eta_{1}],\eta_{2}) +G(\eta_{1},[\frac{dA}{dt}v,\eta_{2}]).
\end{align*}

Part 3) is a standard fact and we omit the proof.

\end{proof}

\begin{lemma} \label{l:variation2}
\begin{align*}
\hspace*{-2.7cm}1) \ 2D_{\cdot_{3}}\frac{dG}{dt}(\cdot_{1},\cdot_{2})D_{\cdot_{3}}G(\cdot_{1},\cdot_{2}) &=<\frac{dG}{dt}, -2(D_{\cdot}DG)_{\cdot} + 4D_{\cdot_{2}}G(\cdot_{1},*)D_{\cdot_{2}}G(\cdot_{1},*)> \\& -2d^{*}(\frac{dG}{dt}(\cdot_{1},\cdot_{2})DG(\cdot_{1},\cdot_{2})).
\end{align*}
2) For $\alpha  \in \Gamma(T^{*}M \otimes \mathfrak{G})$,
\begin{align*}
 G(d^{D}\alpha(\cdot_{1},\cdot_{2}),F(\cdot_{1},\cdot_{2}))
 &=2<\alpha, (d^{D})^{*}F> -2<\alpha(*),G^{-1}D_{\cdot_{1}}G(\diamond,F(\cdot_{1},*))> \\ & \ \ \ -2d^{*}(G(\alpha(\cdot),F(*,\cdot))).
\end{align*}
\end{lemma}

\begin{proof}
 1) Let $\phi= \frac{dG}{dt}(\cdot_{1},\cdot_{2})DG(\cdot_{1},\cdot_{2})$. We will compute $d^{*}\phi$.  Letting $\{\eta_{i}\}$ be an orthonormal frame for $\mathfrak{G}$, first consider 
\begin{align*}
\nabla^{L}\phi &= \nabla^{L}(\frac{dG}{dt}(\eta_{i},\eta_{j})DG(\eta_{i},\eta_{j}))
                       \\ &= D\frac{dG}{dt}(\eta_{i},\eta_{j})DG(\eta_{i},\eta_{j}) 
                              +\frac{dG}{dt}(\eta_{i},\eta_{j})(DDG)(\eta_{i},\eta_{j})
                        \\& \ \ \ +2\frac{dG}{dt}(D\eta_{i},\eta_{j})DG(\eta_{i},\eta_{j}) 
                         +2\frac{dG}{dt}(\eta_{i},\eta_{j})DG(D\eta_{i},\eta_{j})
                         \\& = D\frac{dG}{dt}(\cdot_{1},\cdot_{2})DG(\cdot_{1},\cdot_{2}) 
                                 +\frac{dG}{dt}(\cdot_{1},\cdot_{2})(DDG)(\cdot_{1},\cdot_{2}) 
                            \\ & \ \ \ -2\frac{dG}{dt}(\cdot_{1},\cdot_{2})DG(\cdot_{3},\cdot_{2})DG(\cdot_{3},\cdot_{1}).
\end{align*}

Using this, gives
\begin{align*}
d^{*}\phi &= -\nabla^{L}_{\cdot}\phi(\cdot)
              \\&=-D_{\cdot_{3}}\frac{dG}{dt}(\cdot_{1},\cdot_{2})D_{\cdot_{3}}G(\cdot_{1},\cdot_{2})
               -<\frac{dG}{dt},(D_{\cdot}DG)_{\cdot}>  +2<\frac{dG}{dt},D_{\cdot_{2}}G(\cdot_{1},*)D_{\cdot_{2}}G(\cdot_{1},*)>,
\end{align*}
which proves the first part of the lemma.

2) Letting $\{v_{a}\}$ be an orthonormal frame for $TM$, consider 
\begin{align*}
   G(d^{D}\alpha(\cdot_{1},\cdot_{2}),F(\cdot_{1},\cdot_{2}))
                & =2G(D_{\cdot_{1}}\alpha (\cdot_{2}), F(\cdot_{1},\cdot_{2}))
                 \\ &  =2 v_{a}[G(\alpha(v_{b}),F(v_{a},v_{b}))] -2D_{v_{a}}G(\alpha(v_{b}),  F(v_{a},v_{b}) )
              \\ & \ \ \ -2G(\alpha(\nabla^{L}_{v_{a}}v_{b}), F(v_{a},v_{b}) ) -2G(\alpha(v_{b}), D_{v_{a}}F(v_{a},v_{b}))
                    \\ & \ \ \ -2G( \alpha(v_{b}), F(\nabla^{L}_{v_{a}} v_{a}, v_{b}) + F(v_{a}, \nabla^{L}_{v_{a}}v_{b}))
                \\ & =-2G(\alpha (v_{b}), D_{v_{a}}F(v_{a},v_{b})) -2D_{v_{a}}G(\alpha(v_{b}), F(v_{a},v_{b})) 
                  \\ & \ \ \   -2d^{*}(G(\alpha(v_{a}),F(*,v_{a})   )    )
                \\ &  = 2<\alpha, (d^{D})^{*}F> -2<\alpha(*),G^{-1}D_{\cdot_{1}}G(\diamond,F(\cdot_{1},*))> \\ & \ \ \ -2d^{*}(G(\alpha(\cdot),F(*,\cdot))).
\end{align*}

\end{proof}

\begin{lemma} \label{l:variation3} 
\begin{align*}
& 1)  \frac{d|H|^{2}}{dt}= -3<\dot{g_{_{\mathcal{E}}}}, \mathcal{H}> +2<\dot{H},H>.
 \\& 2) \ d^{*}(\dot{B}(\cdot_{1},\cdot_{2})H(\pi_{TM}*,\cdot_{1},\cdot_{2})) = -D_{\cdot_{3}}\dot{B}(\cdot_{1},\cdot_{2})H(\cdot_{3},\cdot_{1},\cdot_{2})  - \dot{B}(\cdot_{1},\cdot_{2})D_{\cdot_{3}}H(\cdot_{3},\cdot_{1},\cdot_{2}) \\ & \hspace{5.4cm}+2 \dot{B}(\cdot_{1},\cdot_{3})H(\cdot_{4},\cdot_{2},\cdot_{3})D_{\cdot_{4}}G(\cdot_{1},\cdot_{2}).
\\& 3) <d\dot{B},H>= -3<\dot{B}, -d^{*}H + i_{q}H> -3d^{*}(\dot{B}(\cdot_{1},\cdot_{2})H(*,\cdot_{1},\cdot_{2})).
 \end{align*} 
\end{lemma}

\begin{proof}
1) First note that for $m =\pi(p) \in M$, $|\bar{H}|^{2}_{\bar{g}}|_{p} = |H|^{2}_{g_{_{\mathcal{E}}}}|_{m}$ and hence
\begin{align*}
\frac{d|H|^{2}_{g_{_{\mathcal{E}}}}}{dt}|_{m}&= \frac{d|\bar{H}|^{2}_{\bar{g}}}{dt}|_{p}
                \\&  = -3<\frac{d\bar{g}}{dt}, \bar{\mathcal{H}}>|_{p} +2<\frac{d\bar{H}}{dt},\bar{H}>|_{p} 
                  =  -3<\dot{g_{_{\mathcal{E}}}}, \mathcal{H}>|_{m} +2<\dot{H},H>|_{m}.
\end{align*}

2) Let $\phi =\dot{B}(\cdot_{1},\cdot_{2})H(\pi_{TM}*,\cdot_{1},\cdot_{2})$, $v \in \Gamma(TM)$ and $\{e_{m}\}=\{\eta_{i}, v_{a}\}$ be an orthonormal frame for $\mathcal{E}=\mathfrak{G} \oplus TM$. First consider 
  \begin{align*}
(\nabla^{L}\phi)[v]&= \nabla^{L}(\dot{B}(e_{m},e_{n})H(v,e_{m},e_{n})) -\dot{B}(e_{m},e_{n})H(\nabla^{L}v,e_{m},e_{n})
     \\ &   =D\dot{B}(e_{m},e_{n})H(v,e_{m},e_{n}) + \dot{B}(e_{m},e_{n})DH(v,e_{m},e_{n})  
           +\dot{B}(De_{m},e_{n})H(v,e_{m},e_{n}) 
           \\ & \ \ \ +\dot{B}(e_{m},e_{n})H(v,De_{m},e_{n}) +\dot{B}(e_{m},De_{n})H(v,e_{m},e_{n})
           +\dot{B}(e_{m},e_{n})H(v,e_{m},De_{n})
           \\& = D\dot{B}(\cdot_{1},\cdot_{2})H(v,\cdot_{1},\cdot_{2}) +\dot{B}(\cdot_{1},\cdot_{2})DH(v,\cdot_{1},\cdot_{2})
                -2\dot{B}(\cdot_{1}, \cdot_{3}) H(v,\cdot_{2}, \cdot_{3})DG(\cdot_{1}, \cdot_{2}).
\end{align*}
It then follows that
\begin{align*}
d^{*}\phi &= -\nabla^{L}_{\cdot}\phi(\cdot) \\&= -D_{\cdot_{3}}\dot{B}(\cdot_{1},\cdot_{2})H(\cdot_{3},\cdot_{1},\cdot_{2}) -\dot{B}(\cdot_{1},\cdot_{2})D_{\cdot_{3}}H(\cdot_{3},\cdot_{1},\cdot_{2})
                +2\dot{B}(\cdot_{1}, \cdot_{3}) H(\cdot_{4},\cdot_{2}, \cdot_{3})D_{\cdot_{4}}G(\cdot_{1}, \cdot_{2}).
\end{align*}
\begin{align*}
3) <d\dot{B},H>&= d\dot{B}(\cdot_{1},\cdot_{2},\cdot_{3})H(\cdot_{1},\cdot_{2},\cdot_{3})
                     \\ &= (3D_{\cdot_{1}}\dot{B}(\cdot_{2},\cdot_{3})
                         +3\dot{B}(F(\cdot_{1},\cdot_{2}),\cdot_{3})
                          -3\tr_{G}^{(1,2)}\dot{B}([\cdot_{1},\cdot_{2}],\cdot_{3}))H(\cdot_{1},\cdot_{2},\cdot_{3})
                     \\&= -3 d^{*}(\dot{B}(\cdot_{1},\cdot_{2})H(\pi_{TM}*,\cdot_{1},\cdot_{2}))  
                            -3<\dot{B}, D_{\cdot}H(\cdot, *,*) > 
                           \\& \ \ \   +\dot{B}(\cdot_{1},\cdot_{2})(6D_{\cdot_{3}}G(\cdot_{1},\cdot_{4} )
                       +3G(F(\cdot_{3},\cdot_{4}),\cdot_{1}) 
              -3\tr_{G}^{(3,4)}G([\cdot_{3},\cdot_{4}],\cdot_{1})) H(\cdot_{3},\cdot_{4},\cdot_{2} )
             \\ & =-3<\dot{B}, -d^{*}H + i_{q}H> -3d^{*}(\dot{B}(\cdot_{1},\cdot_{2})H(*,\cdot_{1},\cdot_{2})),
                      \end{align*}                       
where we used Part 2) and Propositions \ref{PROPDB1} and \ref{p:dstarHdecomp}.

\end{proof}

Putting these lemmas together we can compute the variation of the different components of the density of $\FF$.

\begin{prop}
\begin{align*}
\hspace{-.8cm}  1) \frac{d}{dt}|DG|^{2}&=  -2<\frac{dG}{dt},(D_{\cdot}DG)_{\cdot}- D_{\cdot_{2}}G(*,\cdot_{1})D_{\cdot_{2}}G(*,\cdot_{1})> \\& \ \ \ + 4 <\frac{dA}{dt}(*),D_{*}G(\cdot_{1},\cdot_{2})G^{-1}G([\diamond, \cdot_{1}],\cdot_{2})> -<\frac{dg}{dt},DG(\cdot_{1},\cdot_{2})DG(\cdot_{1},\cdot_{2})> \\& \ \ \ -2d^{*}(DG(\cdot_{1},\cdot_{2})\frac{dG}{dt}(\cdot_{1},\cdot_{2})).
\end{align*}
\begin{align*}
2) \frac{d}{dt}|F|^{2}&= <\frac{dG}{dt}, G(F(\cdot_{1},\cdot_{2}),*)G(F(\cdot_{1},\cdot_{2}),*)> +4<\frac{dA}{dt}(*), (d^{D})^{*}F -G^{-1}D_{\cdot_{1}}G(\diamond,F(\cdot_{1},*))> \\&-2<\frac{dg}{dt},G(F(*,\cdot),F(*,\cdot))>  -4d^{*}(G(\frac{dA}{dt}(\cdot),F(*,\cdot))).
\end{align*}

\begin{align*}
\hspace{-1.6cm}  3) \frac{d|H|^{2}}{dt}e^{-f} &= -3<\frac{dG}{dt},\mathcal{H}(\pi_{\mathfrak{G}}*, \pi_{\mathfrak{G}}*)>e^{-f}
  								  -6<\frac{dA}{dt}(*), G^{-1}\mathcal{H}(\pi_{\mathfrak{G}}\diamond, \pi_{TM}*)> e^{-f}
									 \\& \ \ \ -3 <\frac{dg}{dt}, \mathcal{H}(\pi_{TM}*, \pi_{TM}*) > e^{-f}
									     -6<\dot{B}, -d^{*}H +i_{q-\nabla f}H>e^{-f}
					 \\& \ \ \   -6 d^{*}(e^{-f} \dot{B}(\cdot_{1},\cdot_{2})H(\pi_{TM}*,\cdot_{1},\cdot_{2}  )   ).
\end{align*}
\begin{align*}
\hspace{-4.2cm} 4) \frac{d|[,]|^{2}}{dt}= <\frac{dG}{dt}, -2\tr_{G}G([*,\cdot],[*,\cdot]) +\tr_{G}G([\cdot_{1},\cdot_{2}],*)G([\cdot_{1},\cdot_{2}],*)>.
\end{align*}

\end{prop}

\begin{proof}
We will be using Lemmas \ref{l:variation1}, \ref{l:variation2} and \ref{l:variation3} throughout without referencing.
\begin{align*}
1) \frac{d}{dt}|DG|^{2} &= -2<\frac{dG}{dt}, D_{\cdot_{2}}G(*,\cdot_{1})D_{\cdot_{2}}G(*,\cdot_{1})> -<\frac{dg}{dt},D_{*}G(\cdot_{1},\cdot_{2})D_{*}G(\cdot_{1},\cdot_{2})> \\& \ \ \  + 2(\frac{dDG}{dt})_{\cdot_{3}}(\cdot_{1},\cdot_{2})D_{\cdot_{3}}G(\cdot_{1},\cdot_{2})
\\ &= -2<\frac{dG}{dt}, D_{\cdot_{2}}G(*,\cdot_{1})D_{\cdot_{2}}G(*,\cdot_{1})> -<\frac{dg}{dt}, D_{*}G(\cdot_{1},\cdot_{2})D_{*}G(\cdot_{1},\cdot_{2})> \\& \ \ \ + 2(D_{\cdot_{3}}\frac{dG}{dt})(\cdot_{1},\cdot_{2})D_{\cdot_{3}}G(\cdot_{1},\cdot_{2}) +4G([\frac{dA}{dt}\cdot_{3},\cdot_{1}],\cdot_{2})D_{\cdot_{3}}G(\cdot_{1},\cdot_{2})
\\ & = -2<\frac{dG}{dt},(D_{\cdot}DG)_{\cdot}- D_{\cdot_{2}}G(*,\cdot_{1})D_{\cdot_{2}}G(*,\cdot_{1})> \\& \ \ \ + 4 <\frac{dA}{dt}(*),D_{*}G(\cdot_{1},\cdot_{2})G^{-1}G([\diamond, \cdot_{1}],\cdot_{2})> -<\frac{dg}{dt},DG(\cdot_{1},\cdot_{2})DG(\cdot_{1},\cdot_{2})> \\& \ \ \ -2d^{*}(DG(\cdot_{1},\cdot_{2})\frac{dG}{dt}(\cdot_{1},\cdot_{2})).
\end{align*}
\begin{align*}
2) \frac{d}{dt}|F|^{2}=& \frac{d}{dt} G(F(\cdot_{1},\cdot_{2}), F(\cdot_{1},\cdot_{2}))
      \\ & = -2 <\frac{dg}{dt}, G(F(*,\cdot), F(*,\cdot))> 
          +\frac{dG}{dt}( F(\cdot_{1},\cdot_{2}), F(\cdot_{1},\cdot_{2}) )
           +2G( d^{D}\frac{dA}{dt}(\cdot_{1},\cdot_{2}),F(\cdot_{1},\cdot_{2}) ).
        \end{align*}
Using Part 2) of  Lemma \ref{l:variation2} for $\alpha= \frac{dA}{dt}$ proves Part 2) of the proposition.
\begin{align*}
3) \frac{d|H|^{2}}{dt}e^{-f}&= -3<\dot{g}_{_{\mathcal{E}}}, \mathcal{H}>e^{-f} +2<\dot{H},H>e^{-f}
  										 \\&= -3<\frac{dG}{dt},\mathcal{H}(\pi_{\mathfrak{G}}*, \pi_{\mathfrak{G}}*)>e^{-f}
  								  -6<\frac{dA}{dt}(*), G^{-1}\mathcal{H}(\pi_{\mathfrak{G}}\diamond, \pi_{TM}*)> e^{-f}
									 \\& \ \ \ -3 <\frac{dg}{dt}, \mathcal{H}(\pi_{TM}*, \pi_{TM}*) > e^{-f}
									     -6<\dot{B}, -d^{*}H +i_{q-\nabla f}H>e^{-f}
					 \\& \ \ \   -6 d^{*}(e^{-f} \dot{B}(\cdot_{1},\cdot_{2})H(\pi_{TM}*,\cdot_{1},\cdot_{2}  )   ).
\end{align*}
 For Part 4), the claim follows easily noting that $[,]$ on $\Gamma(\mathfrak{G})$ is independent of $t$.
\end{proof}

\begin{prop} \label{p:Fvariation} Given $(\bar{g}(t),\bar{H}(t),f(t))$ such that $\dot{H}(t)=d\dot{B}(t)$, for $\dot{B}(t) \in \Gamma(\wedge^{2}\mathcal{E}^{*})$, we have
\begin{align*}
&\frac{d\mathcal{F}}{dt}(\bar{g}(t),\bar{H}(t),f(t))= 
\\ &\int_{M} <\frac{dG}{dt}, \frac{1}{2}(D_{\cdot}DG)_{\cdot} -\frac{1}{2}D_{\cdot_{2}}G(*,\cdot_{1}) D_{\cdot_{2}}G(*,\cdot_{1}) -\frac{1}{4}G(F(\cdot_{1},\cdot_{2}),*) G(F(\cdot_{1},\cdot_{2}),*) \\ & \hspace*{10mm}   +\frac{1}{2} \tr_{G}G([*,\cdot], [*,\cdot]) -\frac{1}{4}\tr_{G}^{(1,2)} G([\cdot_{1},\cdot_{2}],*) G([\cdot_{1},\cdot_{2}],*)+\frac{1}{4}\mathcal{H}(\pi_{\mathfrak{G}}*, \pi_{\mathfrak{G}}* ) -\frac{1}{2}D_{\nabla f}G >e^{-f}dV_{g}
\\ & + \int_{M} < \frac{dA}{dt}(*), -(d^{D})^{*}F +G^{-1} D_{\cdot_{1}}G(\diamond, F(\cdot_{1}, *) ) -G^{-1} G([\diamond,\cdot_{1}],\cdot_{2}) D_{*}G(\cdot_{1},\cdot_{2}) \\ &  \hspace*{10mm} +\frac{1}{2} G^{-1}  \mathcal{H} (\pi_{\mathfrak{G}} \diamond, \pi_{TM}*) - F(\nabla f, *)  > e^{-f}dV_{g}
\\ & + \int_{M} <\frac{dg}{dt}, -Ric_{g} +\frac{1}{4} D_{*}G(\cdot_{1},\cdot_{2})D_{*}G(\cdot_{1},\cdot_{2}) +\frac{1}{2}G(F(*,\cdot), F(*,\cdot))  \\ &  \hspace*{10mm}  +\frac{1}{4}\mathcal{H}(\pi_{TM}*, \pi_{TM}* )  -\N^2 f >  e^{-f}dV_{g}
\\ & + \frac{1}{2}\int_{M}< \dot{B}, -d^{*}H +i_{q-\nabla f}H>  e^{-f}dV_{g}
\\ & +\int_{M} (\frac{1}{2}\tr_{g}\frac{dg}{dt} -\frac{df}{dt})(2\Delta f -|\nabla f|^{2} +R_{g} -\frac{1}{4}|DG|^{2} -\frac{1}{4}|F|^{2} -\frac{1}{12}|H|^{2} -\frac{1}{4}|[,]|^{2})e^{-f}dV_{g}.
\end{align*}
 
\end{prop}

\begin{prop} \label{p:Fmonotonicity}
Suppose $(\bar{g}(t),\bar{H}(t))$ is a solution of generalized Ricci flow, and $u_t = e^{-f_t}$ is a solution of the conjugate heat equation (\ref{f:conjugateheat}).  Then
\begin{gather} \label{f:Fmoneqn}
\begin{split}
&\frac{d\FF}{dt}(\bar{g},\bar{H},f)= \\ 
&\ \frac{1}{2}\int_{M} | (D_{\cdot}DG)_{\cdot} -D_{\cdot_{2}}G(*,\cdot_{1}) D_{\cdot_{2}}G(*,\cdot_{1}) -\frac{1}{2}G(F(\cdot_{1},\cdot_{2}),*) G(F(\cdot_{1},\cdot_{2}),*) \\ & \hspace*{10mm}   + \tr_{G}G([*,\cdot], [*,\cdot])   -\frac{1}{2}\tr_{G}^{(1,2)} G([\cdot_{1},\cdot_{2}],*) G([\cdot_{1},\cdot_{2}],*)  +\frac{1}{2}\mathcal{H}(\pi_{\mathfrak{G}}*, \pi_{\mathfrak{G}}* ) -D_{\nabla f}G|^{2} e^{-f}dV_{g}\\
&\ + \int_{M} | -(d^{D})^{*}F +G^{-1} D_{\cdot_{1}}G(\diamond, F(\cdot_{1}, *) ) +G^{-1}\tr_{G}D_{*}G([\cdot,\diamond],\cdot) +\frac{1}{2} G^{-1} \mathcal{H} (\pi_{\mathfrak{G}} \diamond, \pi_{TM}*)  \\ &  \hspace*{10mm} - F(\nabla f, *)|^{2} e^{-f}dV_{g}\\
&\ + \frac{1}{2}\int_{M}  | -2Ric_{g} +\frac{1}{2} D_{*}G(\cdot_{1},\cdot_{2})D_{*}G(\cdot_{1},\cdot_{2}) +G(F(*,\cdot), F(*,\cdot))  \\ &  \hspace*{10mm}  +\frac{1}{2}\mathcal{H}(\pi_{TM}*, \pi_{TM}* ) -2\N^2 f |^{2} e^{-f}dV_{g}\\
&\ + \frac{1}{2}\int_{M} | D_{\cdot}H(\cdot,*,*)-D_{\cdot_{1}}G(\cdot_{2},\pi_{\mathfrak{G}}*)\wedge H(\cdot_{1},\cdot_{2},*)-\frac{1}{2}G(F(\cdot_{1},\cdot_{2}),\pi_{\mathfrak{G}}*)\wedge H(\cdot_{1},\cdot_{2},*)\\ &  \hspace*{10mm} +\frac{1}{2}\tr_{G}G([\cdot_{1},\cdot_{2}],\pi_{\mathfrak{G}}*)\wedge H(\cdot_{1},\cdot_{2},*)- i_{\nabla f}H|^{2}  e^{-f}dV_{g}.
\end{split}
\end{gather}

\end{prop}

\subsection{Expander Entropy}

\begin{defn} \label{d:expent} Given $\bar{g}, \bar{H}$ as above and $f \in C^{\infty}(M)$, $t \in \mathbb R^+$, let
\begin{gather}
\begin{split}
\WW_+&(\bar{g},\bar{H},f,t)\\
=&\ \frac{1}{(4\pi t)^{\frac{n}{2}}} \left\{ t \mathcal{F}(\bar{g},\bar{H},f) + \int_{M} (-f + n) e^{-f} dV_g \right\}\\
=&\ \int_M \left\{ t \left( |\N f|^{2}+R_{g} -\tfrac{1}{4} |DG|^{2} -\tfrac{1}{4} |F|^{2} -\frac{1}{12} |\bar{H}|_{\bar{g}}^{2} -\tfrac{1}{4} \brs{[,]}^2 \right) - f + n \right\} \frac{e^{-f}}{(4 \pi t)^{\frac{n}{2}}} dV_g.
\end{split}
\end{gather}
\end{defn}

\begin{rmk} In the case $G = \{e\}$, so that $P = M$, we obtain
\begin{align*}
\WW_+(\bar{g},\bar{H},f,t) =&\ \int_M \left[ t \left( \brs{\N f}^2 + \bar{R} - \tfrac{1}{12} \brs{\bar{H}}_{\bar{g}}^2 \right) - f + n \right] e^{-f} dV_g,
\end{align*}
which is the expanding entropy functional for generalized Ricci flow studied in \cite{Streetsexpent}.  On the other hand in the case $H \equiv 0$ and abelian structure group, this functional is precisely that introduced by Lott \cite{Lott2}, which in turn is the functional of Feldman-Ilmanen-Ni \cite{FIN} in the case $G = \{e\}$.
\end{rmk}

\begin{prop} \label{p:Wvariation} Given $(\bar{g}(t),\bar{H}(t),f(t), \tau(t))$ such that $\dot{H}(t)=d\dot{B}(t)$, for  $\dot{B}(t) \in \Gamma(\wedge^{2}\mathcal{E}^{*})$, we have
\begin{align*}
&\frac{d\mathcal{W}_{+}}{dt}(\bar{g}(t),\bar{H}(t),f(t),\tau(t))= 
\\ &\int_{M} <\tau \frac{dG}{dt}, \frac{1}{2}(D_{\cdot}DG)_{\cdot} -\frac{1}{2}D_{\cdot_{2}}G(*,\cdot_{1}) D_{\cdot_{2}}G(*,\cdot_{1}) -\frac{1}{4}G(F(\cdot_{1},\cdot_{2}),*) G(F(\cdot_{1},\cdot_{2}),*) \\ & \hspace*{10mm}   +\frac{1}{2} \tr_{G}G([*,\cdot], [*,\cdot]) -\frac{1}{4}\tr_{G}^{(1,2)} G([\cdot_{1},\cdot_{2}],*) G([\cdot_{1},\cdot_{2}],*)+\frac{1}{4}\mathcal{H}(\pi_{\mathfrak{G}}*, \pi_{\mathfrak{G}}* ) -\frac{1}{2}D_{\nabla f}G  > \frac{e^{-f}dV_{g}}{(4\pi \tau)^{\frac{n}{2}}}
\\ & + \int_{M} <\tau \frac{dA}{dt}(*), -(d^{D})^{*}F +G^{-1} D_{\cdot_{1}}G(\diamond, F(\cdot_{1}, *) ) -G^{-1} G([\diamond,\cdot_{1}],\cdot_{2}) D_{*}G(\cdot_{1},\cdot_{2}) \\ &  \hspace*{10mm} +\frac{1}{2} G^{-1} \mathcal{H} (\pi_{\mathfrak{G}} \diamond, \pi_{TM}*) - F(\nabla f, *)  > \frac{e^{-f}dV_{g}}{(4\pi \tau)^{\frac{n}{2}}}
\\ & + \int_{M} <\tau \frac{dg}{dt}-\frac{d\tau}{dt}g, -\frac{1}{2\tau}g -Ric_{g}  +\frac{1}{4} D_{*}G(\cdot_{1},\cdot_{2})D_{*}G(\cdot_{1},\cdot_{2}) +\frac{1}{2}G(F(*,\cdot), F(*,\cdot))  \\ &  \hspace*{10mm}  +\frac{1}{4}\mathcal{H}(\pi_{TM}*, \pi_{TM}* ) -\N^2 f>  \frac{e^{-f}dV_{g}}{(4\pi \tau)^{\frac{n}{2}}}
\\ & + \frac{1}{2}\int_{M}<\tau \dot{B}, -d^{*}H +i_{q-\nabla f}H>  \frac{e^{-f}dV_{g}}{(4\pi \tau)^{\frac{n}{2}}} 
\\ & + \int_{M} \frac{d\tau}{dt}(\frac{1}{4}|F|^{2} -\frac{1}{4}|[,]|^{2}+ \frac{1}{6}|H|^{2}-\frac{1}{4}\tr_{G}\mathcal{H}(\cdot,\cdot)) \frac{e^{-f}dV_{g}}{(4\pi \tau)^{\frac{n}{2}}}
\\ & + \int_{M} [\tau(2\Delta f -|\nabla f|^{2} +R_{g} -\frac{1}{4} |DG|^{2} -\frac{1}{4} |F|^{2} -\frac{1}{12} |H|^{2} -\frac{1}{4} \tr_{G}^{(1,2)}G([\cdot_{1},\cdot_{2}],[\cdot_{1},\cdot_{2}])) \\ & \hspace*{10mm} +(-f+n+1)] \frac{d}{dt} \frac{e^{-f}dV_{g}}{(4\pi \tau)^{\frac{n}{2}}}.
\end{align*}

\end{prop}

\begin{prop} \label{p:Wmonotonicity}
Suppose $(\bar{g}(t),\bar{H}(t))$ is a solution of generalized Ricci flow, and $u_t = \frac{e^{-f_t}}{(4\pi t)^{\frac{n}{2}}}$ is a solution of the conjugate heat equation (\ref{f:conjugateheat}).  Then
\begin{align*}
&\frac{dW_{+}}{dt}(\bar{g},\bar{H},f,t)= 
\\ &\frac{t}{2}\int_{M} | (D_{\cdot}DG)_{\cdot} -D_{\cdot_{2}}G(*,\cdot_{1}) D_{\cdot_{2}}G(*,\cdot_{1}) -\frac{1}{2}G(F(\cdot_{1},\cdot_{2}),*) G(F(\cdot_{1},\cdot_{2}),*) \\ & \hspace*{10mm}   + \tr_{G}G([*,\cdot], [*,\cdot])   -\frac{1}{2}\tr_{G}  G([\cdot_{1},\cdot_{2}],*) G([\cdot_{1},\cdot_{2}],*)  +\frac{1}{2}\mathcal{H}(\pi_{\mathfrak{G}}*, \pi_{\mathfrak{G}}* ) -D_{\nabla f}G|^{2} \frac{e^{-f}dV_{g}}{(4\pi t)^{\frac{n}{2}}}
\end{align*}
\begin{align*}
 & + t\int_{M} | -(d^{D})^{*}F +G^{-1} D_{\cdot}G(\diamond, F(\cdot, *) ) +G^{-1}\tr_{G}D_{*}G([\cdot,\diamond],\cdot) +\frac{1}{2} G^{-1}  \mathcal{H} (\pi_{\mathfrak{G}} \diamond, \pi_{TM}*)  \\ &  \hspace*{10mm} - F(\nabla f, *)|^{2} \frac{e^{-f}dV_{g}}{(4\pi t)^{\frac{n}{2}}}
\end{align*}
\begin{align*}
& + \frac{t}{2}\int_{M} | -\frac{g}{t} -2Ric_{g} +\frac{1}{2} D_{*}G(\cdot_{1},\cdot_{2})D_{*}G(\cdot_{1},\cdot_{2}) +G(F(*,\cdot), F(*,\cdot))  \\ &  \hspace*{10mm}  +\frac{1}{2}\mathcal{H}(\pi_{TM}*, \pi_{TM}* ) -2\N^2 f |^{2} \frac{e^{-f}dV_{g}}{(4\pi t)^{\frac{n}{2}}}
\\& + \frac{t}{2}\int_{M} | D_{\cdot}H(\cdot,*,*)-D_{\cdot_{1}}G(\cdot_{2},\pi_{\mathfrak{G}}*)\wedge H(\cdot_{1},\cdot_{2},*)-\frac{1}{2}G(F(\cdot_{1},\cdot_{2}),\pi_{\mathfrak{G}}*)\wedge H(\cdot_{1},\cdot_{2},*)\\ &  \hspace*{10mm} +\frac{1}{2}\tr_{G}G([\cdot_{1},\cdot_{2}],\pi_{\mathfrak{G}}*)\wedge H(\cdot_{1},\cdot_{2},*)- i_{\nabla f}H|^{2}  \frac{e^{-f}dV_{g}}{(4\pi t)^{\frac{n}{2}}} 
\\ & + \int_{M} (\frac{1}{4}|F|^{2} -\frac{1}{4}|[,]|^{2}+ \frac{1}{6}|H|^{2}-\frac{1}{4}\tr_{G}\mathcal{H}(\cdot,\cdot)) \frac{e^{-f}dV_{g}}{(4\pi t)^{\frac{n}{2}}}.
\end{align*}

\end{prop}

\section{Rigidity and classification Results} \label{s:rigidity}
\subsection{Convergence}

\begin{defn} \label{d:convergence} Let $\{(P^i, p_i, \bar{g}_i(\cdot), \bar{H}_i(\cdot))\}$ be a sequence of pointed locally $\GG$-invariant solutions to generalized Ricci flow defined on a time interval $(\ga, \gw)$ with $\ga, \gw \in [-\infty, \infty]$.  We say that
\begin{align*}
\lim_{i \to \infty} (P_i, p_i, \bar{g}_i(\cdot), \bar{H}_i(\cdot)) = (P_{\infty}, p_{\infty}, \bar{g}_{\infty}(\cdot), \bar{H}_{\infty}(\cdot))
 \end{align*}
 if there exists a sequence of open neighborhoods of $p_{\infty} \in B_{\infty}$, denoted $\{U_j\}$, and subintervals of $(\ga, \gw)$, denoted $\{I_j\}$, such that
 \begin{enumerate}
 \item For every compact set $K \subset B_{\infty}$ there exists $j_0$ such that $K \subset U_j$ for all $j \geq j_0$.
 \item For all $i, j$ there exists an open neighborhood of $p_i \in B_i$, denoted $V_{i,j}$, together with isomorphisms
 \begin{align*}
 \phi_{i,j} : \pi_{\infty}^{-1}(U_j) \to \pi_i^{-1}(V_{i,j})
 \end{align*}
 such that for all $j$,
 \begin{align*}
 \lim_{i \to \infty} \phi_{i,j}^* \bar{g}_i(\cdot) = \bar{g}_{\infty}(\cdot), \qquad  \lim_{i \to \infty} \phi_{i,j}^* \bar{H}_i(\cdot) = \bar{H}_{\infty}(\cdot)
 \end{align*}
 smoothly on $\pi_{\infty}^{-1}(U_j) \times I_j$
   \end{enumerate}
\end{defn}

\subsection{Energy functional case}

\begin{prop} \label{p:energyrigidity} Suppose $(\bar{g}_t, \bar{H}_t)$ is an invariant solution of generalized Ricci flow, and suppose $u_t = e^{-f_t}$ satisfies (\ref{f:conjugateheat}).  Furthermore suppose $[H] = 0$ and $\GG$ is abelian.
Then $\FF(\bar{g}_t, \bar{H}_t,f_t,t)$ is constant in $t$ if and only if $H \equiv 0$, $F \equiv 0$, $\det G_{ij}$ is constant and $\bar{g}$ is a local product metric with $g$ Ricci flat.
\begin{proof} From Proposition \ref{p:Fmonotonicity}, we know that $\frac{d \FF}{dt}$ is a sum of nonnegative terms, and so is zero if and only if all the terms on the right hand side of (\ref{f:Fmoneqn}) vanish.  From the vanishing of the final term of (\ref{f:Fmoneqn}) we obtain for instance at time zero
\begin{align*}
0 = - d^* H + i_{q - \N f} H.
\end{align*}
Expressing $q = \N \ln \sqrt{\det G_{ij}}$ this can be rewritten as
\begin{align*}
0 =&\ - d^* (e^{\ln \sqrt{\det G_{ij}} - f} H)
\end{align*}
Since $[H] = 0$, we have $H = d b$, thus taking the inner product of the above equation with $b$, integrating over $M$ then integrating by parts yields $H \equiv 0$.  Thus the generalized Ricci flow is in fact a standard Ricci flow with $\GG$ abelian and constant $\FF$ functional, thus the remaining claims follow by (\cite{Lott2} Proposition 4.21, Proposition 4.38).
\end{proof}
\end{prop}

\begin{rmk} \label{r:steadyrmk} The hypothesis that $[H] = 0$ is necessary to obtain the vanishing of $H$ and hence that the resulting structure is actually a solution of the reduced Ricci soliton equations.  In particular, it was shown in \cite{Streetssolitons} that nontrivial steady solitons for the generalized Ricci flow (in fact pluriclosed flow) exist on Hopf surfaces $S^3 \times S^1$.  In the case of elliptic Hopf surfaces these structures are invariant under a principal $T^2$ action with base space a bad orbifold.  For these examples $[H] \neq 0$, and the solitons are moreover nontrivial in the sense that $\N f \neq 0$.
\end{rmk}

\begin{cor} \label{c:LTBviaenergy} Suppose $(P, \bar{g}_t, \bar{H}_t)$ is a $\GG$-invariant generalized Ricci flow on $[0,\infty)$.  Given $\{s_i\}_{i=1}^{\infty}$ a sequence such that $\lim_{i \to \infty} s_i = \infty$, let
\begin{align*}
\bar{g}_i(t) := \bar{g}(s_i + t)\\
\bar{H}_i(t) := \bar{H}(s_i + t).
\end{align*}
Suppose $\lim_{i \to \infty} (\bar{g_i}(\cdot), \bar{H}_i(\cdot)) = (\bar{g}_{\infty}(\cdot), \bar{H}_{\infty}(\cdot))$, defined on a $\GG$-principal bundle $P_{\infty}$ over a compact base $M_{\infty}$.  Furthermore suppose that $[\bar{H}_{\infty}] = 0$ and $\GG$ is abelian.  Then $\bar{H}_{\infty} = 0, F_{\infty} = 0$, and $\bar{g}_{\infty}$ is a local product metric with $g_{\infty}$ Ricci flat.
\begin{proof} The first step is to construct a positive solution of the conjugate heat equation on $P_{\infty}$.  
Fix a sequence of times $\{t_j\} \to \infty$, and let $v_j(\cdot)$ be the unique solution to (\ref{f:conjugateheat}) satisfying $v_j(t_j) = \Vol(g(t_j))^{-1}$.  Note that since $v_j(t_j) > 0$ it follows by the maximum principle that $v_j > 0$ where defined.  On any compact time interval $[0, T]$, equation (\ref{f:conjugateheat}) is a uniformly parabolic scalar PDE with bounds on the derivatives of all coefficients, and so from parabolic regularity theory, for any solution $v_j(\cdot)$ defined on $[0,T]$, there exist uniform $C^{\infty}$ estimates on $[0,T-1]$.  Thus for any $T$ we can obtain a subsequence of $\{v_j(\cdot)\}$ converging uniformly on $[0,T-1]$, and by taking a diagonal subsequence as $T \to \infty$, we obtain a subsequence of $\{v_j(\cdot)\}$ which converges smoothly on compact subsets of $[0,\infty)$ to a solution $v_{\infty}(\cdot)$ defined on $[0,\infty)$, which is nontrivial since all $\{v_j(\cdot)\}$ have unit mass with respect to $g(\cdot)$.  We furthermore claim that $v_{\infty}(\cdot) > 0$.  Fix a time $t_0 \in (0,\infty)$ and fix $p \in B_{\infty}$.  Using the positive heat kernel for the time-dependent Laplacian $\gD_{g_{\infty}(t)}$, we can construct a solution to the heat equation
\begin{align*}
\left(\dt - \gD_{g_{\infty}(t)} \right)f =&\ 0 \qquad \mbox{on}\qquad B_{\infty} \times (t_0, t_0+1]\\
\lim_{t \to t_0} f =&\ \gd_{p}.
\end{align*}
By the maximum principle, since $B_{\infty}$ is compact $f(x, t_0 + 1)$ has a strict positive lower bound.  We compute
\begin{align*}
\frac{d}{dt} \int_{B_{\infty}} f v dV_g =&\ \int_{B_{\infty}} \left\{ \left( \gD f\right) v + f \left( - \gD v + \left(R - \tfrac{1}{4} \brs{DG}^2 - \tfrac{1}{2} \brs{F}^2 - \tfrac{1}{4} \tr_g \HH \right) v \right) \right.\\
&\ \left. \qquad + f v \left(-R + \tfrac{1}{4} \brs{DG}^2 + \tfrac{1}{2} \brs{F}^2 + \tfrac{1}{4} \tr_g \HH \right) \right\} dV_{g}\\
=&\ 0.
\end{align*}
Thus $\int_{B_{\infty}} fv dV_g$ is constant in time, and thus
\begin{align*}
v(p, t_0) = \lim_{t \to t_0} \int_{B_{\infty}} f v dV_g = \int_{B_{\infty} \times \{t_0 + 1\}} f v dV_g \geq \inf_{B_{\infty}} f(\cdot, t_0 + 1) > 0,
\end{align*}
as claimed.

As $v_{\infty}$ is strictly positive, we define $\phi_{\infty}(\cdot)$ via $v_{\infty}(\cdot) = e^{- \phi_{\infty}(\cdot)}$.  Furthermore we set $u_i(t) = v_{\infty}(s_i + t) = e^{-f_i(t)}$.  Using the uniform geometric estimates on $\bar{g}_i(\cdot)$ it follows as above that $u_i$ converges smoothly on compact subsets of $[0, \infty)$ to a solution $u_{\infty}(\cdot)$ of the conjugate heat equation, and we furthermore set $u_{\infty}(\cdot) = e^{-f_{\infty}(\cdot)}$.  By the convergence properties already established, it follows that, for any $t \in [0, \infty)$, 
\begin{align*}
\FF(\bar{g}_{\infty}(t), \bar{H}_{\infty}(t), f_{\infty}(t)) =&\ \lim_{i \to \infty} \left( \bar{g}_i(t), \bar{H}_i(t), f_i(t) \right)\\
=&\ \lim_{i \to \infty} \left( \bar{g}(s_i + t), \bar{H}_i(s_i + t), \phi_{\infty}(s_i + t) \right).
\end{align*}
The final expression is finite and independent of $t$, and so the claims follow from Proposition \ref{p:energyrigidity}.
\end{proof}
\end{cor}

\subsection{Expander entropy case}

To treat the expander entropy case we again must deal with the final terms in the evolution equation for $\WW_+$ which have potentially mixed sign.  We begin with a preliminary lemma.

\begin{lemma} \label{l:partialpositivity} Given $(\bar{g}, \bar{H})$, one has
\begin{align*}
\frac{1}{6}|H|_{g_{\mathcal E}}^{2}-&\frac{1}{4}\tr_{G}\mathcal{H}(\cdot,\cdot)\\
=&\ -\frac{1}{12}\tr_{G}H(\cdot_{1},\cdot_{2},\cdot_{3})H(\cdot_{1},\cdot_{2},\cdot_{3}) + \frac{1}{4}\tr_{G}^{1} \tr_{g}^{(2,3)}H(\cdot_{1},\cdot_{2},\cdot_{3})H(\cdot_{1},\cdot_{2},\cdot_{3})\\
&\ +\frac{1}{6}\tr_{g}H(\cdot_{1},\cdot_{2},\cdot_{3})H(\cdot_{1},\cdot_{2},\cdot_{3}).
\end{align*}
\begin{proof} Direct computation yields
\begin{align*}
\brs{H}^2_{\bar{g}} =&\ \tr_g H(\cdot_{1},\cdot_{2},\cdot_{3})H(\cdot_{1},\cdot_{2},\cdot_{3}) + 3\tr_{G}^{1} \tr_{g}^{(2,3)}H(\cdot_{1},\cdot_{2},\cdot_{3})H(\cdot_{1},\cdot_{2},\cdot_{3})\\
&\ + 3 \tr_{g}^{1} \tr_{G}^{(2,3)}H(\cdot_{1},\cdot_{2},\cdot_{3})H(\cdot_{1},\cdot_{2},\cdot_{3}) + \tr_G H(\cdot_{1},\cdot_{2},\cdot_{3})H(\cdot_{1},\cdot_{2},\cdot_{3}),
\end{align*}
while
\begin{align*}
\tr_G \bar{\HH}(\cdot, \cdot) =&\ \tr_{G}^{1} \tr_{g}^{(2,3)}H(\cdot_{1},\cdot_{2},\cdot_{3})H(\cdot_{1},\cdot_{2},\cdot_{3}) + 2 \tr_{g}^{1} \tr_{G}^{(2,3)}H(\cdot_{1},\cdot_{2},\cdot_{3})H(\cdot_{1},\cdot_{2},\cdot_{3})\\
&\ + \tr_G H(\cdot_{1},\cdot_{2},\cdot_{3})H(\cdot_{1},\cdot_{2},\cdot_{3}).
\end{align*}
Combining these yields the equality.
\end{proof}
\end{lemma}

\begin{prop} \label{p:entropyrigidity} Suppose $(\bar{g}_t, \bar{H}_t)$ is a solution of generalized Ricci flow, and suppose $u_t = \frac{e^{-f_t}}{(4 \pi t)^{\frac{n}{2}}}$ satisfies (\ref{f:conjugateheat}).  Furthermore suppose $\GG$ is abelian and for all $t$
\begin{align} \label{f:nofibertorsion}
\pi_{\Lambda^3 \mathfrak g^*} H = 0.
\end{align}
Then 
\begin{align*}
\frac{d}{dt} \WW_+(\bar{g}_t, \bar{H}_t, f_t, t) \geq 0,
\end{align*}
and $\WW_+(\bar{g}_t, \bar{H}_t,f_t,t)$ is constant in $t$ if and only if $F \equiv 0$, $H \equiv 0$, $\det G_{ij}$ is constant and
\begin{gather} \label{f:Ricciexpandersoliton}
\begin{split}
(D_{\cdot}DG)_{\cdot} -DG_{\cdot_{2}}(*,\cdot_{1}) DG_{\cdot_{2}}(*,\cdot_{1}) =&\ 0\\
\Rc_g - \tfrac{1}{4} D G_{*}(\cdot_1, \cdot_2) D G_{*}(\cdot_1, \cdot_2) + \tfrac{1}{2 t} g =&\ 0.
\end{split}
\end{gather}
\begin{proof} Referring to Proposition \ref{p:Wmonotonicity}), since $\GG$ is assumed abelian the term $-\tfrac{1}{4} \brs{[,]}^2_{G}$ vanishes.  Furthermore, applying Lemma \ref{l:partialpositivity}, and the assumption (\ref{f:nofibertorsion}), it follows that
\begin{align*}
\tfrac{1}{6} \brs{H}^2_{\bar{g}} - \tfrac{1}{4} \tr_G H =&\ \tfrac{1}{4}\tr_{G}^{1} \tr_{g}^{(2,3)}H(\cdot_{1},\cdot_{2},\cdot_{3})H(\cdot_{1},\cdot_{2},\cdot_{3}) +\tfrac{1}{6}\tr_{g}H(\cdot_{1},\cdot_{2},\cdot_{3})H(\cdot_{1},\cdot_{2},\cdot_{3})\\
\geq&\ \tfrac{1}{12} \brs{ \pi_{\Lambda^1(\mathfrak g^*) \wedge \Lambda^2 T^*M} H}_{g_{_{\mathcal E}}}^2 + \tfrac{1}{6} \brs{ \pi_{\Lambda^3 T^*M} H}^2_g.
\end{align*}
In particular, we have shown that the time derivative of $\WW_+$ is a sum of nonnegative terms, which then all must vanish in the case $\WW_+$ is constant in $t$.  In particular, it follows immediately that $F \equiv 0$, and also that $H \in \gG \left( \Lambda^2(\mathfrak g^*) \wedge \Lambda^1 T^*M \right)$, as now all other components must vanish.  Using this and taking the trace of the equation $\frac{\del G}{\del t} = \IP{\N f, D G}$ yields
\begin{align*}
-d^{*}(e^{-f}DG(\cdot,\cdot)) +\tfrac{1}{2} \tr_G \bar{\mathcal H} e^{-f} = 0
\end{align*}
Integrating this equation over $M$, the divergence term will vanish, yielding that $\tr_G \bar{\mathcal H} = 0$, and hence $H$ vanishes completely.  With the vanishing of $H$, we in fact have a solution to Ricci flow, and so the remaining claims can be obtained from (\cite{Lott2} Proposition 4.67).

\end{proof}
\end{prop}

\begin{rmk} \label{e:Inoue} The hypothesis that the torsion $H$ has no vertical piece is necessary to obtain vanishing of $H$ and a reduction to the reduced expanding Ricci soliton equations as described in Proposition \ref{p:entropyrigidity}.  Boling showed (\cite{Boling} Proposition 4.7) that homogeneous solutions to pluriclosed flow on Inoue surfaces of type $S_A$ exist globally.  Furthermore, there is a a universal blowdown limit on the universal cover which is an expanding soliton.  These surfaces are $T^3$ bundles over $S^1$, and the torsion of homogeneous metrics has nonvanishing projection onto $\Lambda^3 \mathfrak g$.
\end{rmk}

\begin{cor} \label{c:LTBviaexpanders} Suppose $(P, \bar{g}(t), \bar{H}(t))$ is a $\GG$-invariant generalized Ricci flow on $[0,\infty)$.  Furthermore suppose $\GG$ is abelian and for all $t$,
\begin{align}
\pi_{\Lambda^3 \mathfrak g^*} H = 0.
\end{align}
Given $\{s_i\}_{i=1}^{\infty}$ a sequence such that $\lim_{i \to \infty} s_i = \infty$, let
\begin{align*}
\bar{g}_i(t) := s_i^{-1} \bar{g}(s_i t)\\
\bar{H}_i(t) := s_i^{-1} \bar{H}(s_i t).
\end{align*}
Suppose $\lim_{i \to \infty} (\bar{g_i}(\cdot), \bar{H}_t(\cdot)) = (\bar{g}_{\infty}(\cdot), \bar{H}_{\infty}(\cdot))$, defined on a $\GG$-principal bundle $P_{\infty}$ over $M_{\infty}$.  Then $F \equiv 0$, $H \equiv 0$, $\det G_{ij}$ is constant and equations (\ref{f:Ricciexpandersoliton}) hold for $\bar{g}_{\infty}$.
\end{cor}

\begin{cor} \label{pcfcor} Suppose $(P, \bar{g}(t), \bar{H}(t))$ is a $\GG$-invariant pluriclosed flow on $[0,\infty)$.  Furthermore suppose $\GG$ is abelian and for all $t$,
\begin{align}
\pi_{\Lambda^3 \mathfrak g^*} H = 0.
\end{align}
With the assumptions of Corollary \ref{c:LTBviaexpanders}, it furthermore follows that $(\bar{g}_{\infty}(t), \bar{J}_{\infty}(t))$ satsifies
\begin{itemize}
\item[1)] $J_{\infty}$ is fixed in time and $(\bar{g}_{\infty}(t), {J}_{\infty})$ is Kahler.
\item[2)] $M_{\infty}$ is even dimensional and $J_{\infty}=J_{1} \oplus J_{2}$ on $\mathfrak{G} \oplus TM_{\infty}$.
\end{itemize}
In particular, if $\dim \GG = 1$ then invariant pluriclosed flow has no subsequential limits in the sense of Definition \ref{d:convergence}.
\begin{proof} Using the rigidity of the limit implied by Corollary \ref{c:LTBviaexpanders}, we see that along the limiting solution of pluriclosed flow is in fact a solution to K\"ahler-Ricci flow, and the complex structure remains fixed in time.  Furthermore, the equations (\ref{f:Ricciexpandersoliton}) imply that the metric on the fiber is fixed while the metric on the base expands homothetically.  Since the metric is also Hermitian for all times, it follows immediately that $J_{\infty}$ must preserve decompose as claimed in Part 2).  The final sentence then follows immediately.
\end{proof}
\end{cor}

\end{document}